\RequirePackage{fix-cm}
\documentclass[numbook]{svjour3}                  
\smartqed 
\usepackage[T1]{fontenc}
\usepackage[utf8]{inputenc}
\usepackage{array}
\usepackage{float}
\usepackage{booktabs}
\usepackage{bm}
\usepackage{units}
\usepackage{multirow}
\usepackage{amsmath}

\usepackage{amssymb}
\usepackage{graphicx}
\usepackage{esint}
\usepackage{lineno}
\usepackage{hyperref}
\usepackage{epstopdf}
\usepackage{cite}
\usepackage{mathtools}
\usepackage{xspace}
\usepackage{enumerate}
\usepackage{subcaption}
\usepackage{nccmath, bigstrut}
\usepackage[misc]{ifsym}

\def\mycmd{1}
\newcommand{\mtrx}[1]{\mathbf{#1}}
\newcommand{\gmtrx}[1]{\bm{#1}}
\newcommand{\me}{\mathrm{e}}
\newcommand{\mi}{\mathrm{i}}

\newcommand{\norm}[1]{\left\lVert#1\right\rVert}
\makeatletter
\renewcommand*\env@matrix[1][c]{\hskip -\arraycolsep
  \let\@ifnextchar\new@ifnextchar
  \array{*\c@MaxMatrixCols #1}}
\makeatother
\begin{document}
\title{Analytic Prediction of Isolated Forced Response Curves from Spectral Submanifolds}
\titlerunning{Analytic Prediction of Isolated Forced Response Curves from Spectral Submanifolds}   
\author{S. Ponsioen \and T. Pedergnana \and G. Haller}
\institute{S. Ponsioen \and T. Pedergnana \and G. Haller (\Letter) \at
              Institute for Mechanical Systems, 
              ETH Z{\"u}rich, 
              Leonhardstrasse 21, 
              8092, Z{\"u}rich, Switzerland \\
              \email{georgehaller@ethz.ch}
}
\date{}
\maketitle
\begin{abstract}
We show how spectral submanifold theory can be used to provide analytic predictions for the response of periodically forced multi-degree-of-freedom mechanical systems. These predictions include an explicit criterion for the existence of isolated forced responses that will generally be missed by numerical continuation techniques. Our analytic predictions can be refined to arbitrary precision via an algorithm that does not require the numerical solutions of the mechanical system. We illustrate all these results on low- and high-dimensional nonlinear vibration problems. We find that our SSM-based forced-response predictions remain accurate in high-dimensional systems, in which numerical continuation of the periodic response is no longer feasible. 


\end{abstract}

\section{Introduction}
For an $n$-degree-of-freedom, periodically forced, nonlinear mechanical system, the forced response curve (FRC) gives the amplitude of the periodic response of the system as a function of the frequency of the periodic forcing. The FRC may contain isolated branches of periodic solutions, also known as \textit{isolas}, that are detached from the main FRC. A small change in the forcing amplitude might result in the merger of the isola with the main branch of the FRC (cf. Detroux et al.\ \cite{Detroux2018} and Noël et al. \cite{noel2015isolated}), which can lead to an unexpected and significant increase in the response amplitude.

The existence of isolated branches of periodic solutions in the frequency response of nonlinear oscillatory systems has been known since the 1950s \cite{abramson1955response}. For an extensive review of the subject, we refer the reader to Habib et al.\ \cite{Habib2017}. It is broadly agreed that the identification of isolas is difficult, because numerical continuation techniques are generally initiated on a non-isolated solution branch and will therefore miss any isolated branch. Similarly, a frequency sweep of the full system will generally not capture an isolated response unless the sweep is initialized on one. 

The detection of isolas and the prediction of their behavior under changing system parameters can be critical in practice because, their merger with the main FRC may lead to a dramatic shift in the resonance frequency and response amplitude. Habib et al. \cite{Habib2017} use singularity theory in combination with averaging for the prediction and identification of isolas in a specific single-degree-of-freedom mechanical system with nonlinear damping. The averaging method they use (cf. Sanders et al. \cite{sanders2007averaging}), however,  requires both the forcing amplitude and the nonlinear damping coefficients to be small. Hill et al. \cite{hill2016analytical} use a second-order normal form technique to obtain analytical expressions for the autonomous conservative backbone curves (i.e. amplitude-frequency plots of nonlinear periodic orbits) of a specific two-degree-of-freedom mechanical system. They give leading-order criteria for the intersection of this backbone curve with the forced response curve and postulate this location to be a potential starting point for an isola, which is to be constructed numerically in a separate effort. This procedure also relies on the smallness of the nonlinear and damping coefficients, as well as on the absence of quadratic nonlinearities.

In summary, while the significance of isolas is broadly recognized, their existence has only been studied in specific, low-dimensional examples under restrictions on the nonlinearities. A conclusive analytical criterion for predicting isolas in multi-degree of freedom systems without costly numerical simulations, therefore, has been unavailable. 

In this work, we seek to fill this gap by developing a generally applicable methodology for the prediction of isolas in multi-degree-of-freedom, forced mechanical systems. Our approach is based on the mathematically rigorous theory of spectral submanifolds (SSMs) that are the unique, smoothest, nonlinear continuations of spectral subspaces of the linearized, unforced limit of a mechanical system (cf. Haller and Ponsioen \cite{Haller2016}). The reduced dynamics on a two-dimensional SSM serves as an exact, single-degree-of-freedom reduced-order model that can be constructed for each vibration mode of the full nonlinear system (cf. \cite{ponsioen2018automated,Jain2018,Breunung2017,Szalai2017,kogelbauer2018rigorous}). 

By construction, these rigorously, simplified two-dimensional reduced models will capture all isolas that are remnants of periodic orbit families of the conservative limit of the system. As we show for a cubic-order approximation, the reduced SSM dynamics gives a closed form first-order prediction for isolas that can even be calculated by hand in simple examples. Higher-order refinements to this analytic formula can be recursively constructed and have been implemented in the publicly available \textsc{MATLAB} script \textsc{SSMtool}\footnote{\textsc{SSMtool} is available at: \href{http://www.georgehaller.com}{www.georgehaller.com}}. We show the use of the analytic formula as well as its numerical refinements on simple and more complicated examples.

\section{System set-up }
We consider $n$-degree-of-freedom, periodically forced mechanical systems of the form
\begin{gather}
\mtrx{M}\ddot{\mtrx{y}}+\mtrx{C}\dot{\mtrx{y}}+\mtrx{K}\mtrx{y}+\mtrx{g}(\mtrx{y},\dot{\mtrx{y}})=\varepsilon \mtrx{f}(\Omega t),\quad 0\leq\varepsilon\ll 1, \label{eq:mech_sys} \\ 
\mtrx{g}(\mtrx{y},\dot{\mtrx{y}})=\mathcal{O}\left(\left|\mtrx{y}\right|^{2},\left|\mtrx{y}\right|\left|\dot{\mtrx{y}}\right|,\left|\dot{\mtrx{y}}\right|^{2}\right), \nonumber 
\end{gather}
where $\mtrx{y}\in\mathbb{R}^{n}$ is the generalized position vector; $\mtrx{M}=\mtrx{M}^{T}\in\mathbb{R}^{n\times n}$ is the positive definite mass matrix; $\mtrx{C}=\mtrx{C}^{T}\in\mathbb{R}^{n\times n}$ is the damping matrix; $\mtrx{K}=\mtrx{K}^{T}\in\mathbb{R}^{n\times n}$ is the stiffness matrix and $\mtrx{g}(\mtrx{y},\dot{\mtrx{y}})$ denotes all the nonlinear terms in the system. These nonlinearities are assumed to be analytic for simplicity. The external forcing $\varepsilon \mtrx{f}(\Omega t)$ does not depend on the positions and velocities. 

System (\ref{eq:mech_sys}) can be transformed into a set of $2n$ first-order ordinary differential equations by introducing the change of variables $\mtrx{x}_{1}=\mtrx{y}$, $\mtrx{x}_{2}=\dot{\mtrx{y}}$, with $\mtrx{x}=(\mtrx{x}_{1},\mtrx{x}_{2})\in\mathbb{R}^{2n}$, which gives
\begin{align}
\dot{\mtrx{x}} & =
\left(\begin{array}{cc}
\mtrx{0} & \mtrx{I}\\
-\mtrx{M}^{-1}\mtrx{K} & -\mtrx{M}^{-1}\mtrx{C}
\end{array}\right)\mtrx{x}
+\left(\begin{array}{c}
\mtrx{0}\\
-\mtrx{M}^{-1}\mtrx{g}(\mtrx{x}_{1},\mtrx{x}_{2})
\end{array}\right)
+ \varepsilon\left(
\begin{array}{c}
\mtrx{0}\\
\mtrx{M}^{-1}\mtrx{f}(\Omega t)
\end{array}
\right) \nonumber \\
&=\mtrx{A}\mtrx{x}+\mtrx{G}_\text{p}(\mtrx{x})+
\varepsilon \mtrx{F}_\text{p}(\Omega t).\label{eq:dyn_sys}
\end{align}
The transformed first-order system (\ref{eq:dyn_sys}) has a fixed point at $\mtrx{x}=\mtrx{0}$ when the system is unforced ($\varepsilon=0$); $\mtrx{A}\in\mathbb{R}^{2n\times2n}$ is a constant matrix and $\mtrx{G}_\text{p}(\mtrx{x})$ is an analytic function containing all the nonlinearities. Note that $\mtrx{M}^{-1}$ is well-defined because $\mtrx{M}$ is assumed positive definite. 

The linearized part of system (\ref{eq:dyn_sys}) is 
\begin{equation}
\dot{\vec{x}}=\vec{A}\vec{x},\label{eq:lin_dyn_sys}
\end{equation}
where the matrix $\vec{A}$ has $2n$ eigenvalues $\lambda_{k}\in\mathbb{C}$
for $k=1,\ldots,2n$. Counting multiplicities, we sort these eigenvalues
based on their real parts in the decreasing order

\begin{equation}
\text{\text{Re}}(\lambda_{2n})\leq\text{\text{Re}}(\lambda_{2n-1})\leq\ldots\leq\text{\text{Re}}(\lambda_{1})<0,
\end{equation}
assuming that the real part of each eigenvalue is less than zero and
hence the fixed point of Eq. (\ref{eq:lin_dyn_sys}) is asymptotically stable. We further assume that the constant matrix $\vec{A}$ is semisimple, i.e., the
algebraic multiplicity of each $\lambda_{k}$ is equal to its geometric
multiplicity: $\text{alg}(\lambda_{k})=\text{geo}(\lambda_{k})$.
We can, therefore, identify $2n$ linearly independent eigenvectors
$\vec{v}_{k}\in\mathbb{C}^{2n}$, with $k=1,\ldots,2n$, each spanning a
real eigenspace $E_{k}\subset\mathbb{R}^{2n}$ with $\text{dim}(E_{k})=2\times\text{alg}(\lambda_{k})$
in case $\text{Im}(\lambda_{k})\neq0$, or $\text{dim}(E_{k})=\text{alg}(\lambda_{k})$
in case $\text{Im}(\lambda_{k})=0$.

\section{Non-autonomous spectral submanifolds and their reduced dynamics}
Because $\mtrx{A}$ is semisimple, the linear part of system (\ref{eq:dyn_sys}) is diagonalized by a linear change of coordinates $\mtrx{x}=\mtrx{T}\mtrx{q}$, with $\mtrx{T}=\left[\mtrx{v}_{1},\mtrx{v}_{2},\ldots,\mtrx{v}_{2n}\right]\in\mathbb{C}^{2n\times2n}$ and $\mtrx{q}\in\mathbb{C}^{2n}$, yielding
\begin{gather}
\dot{\mtrx{q}}=\underbrace{\text{diag}(\lambda_{1},\lambda_{2}\ldots,\lambda_{2n})}_{\vec{\Lambda}}\mtrx{q}+\mtrx{G}_\text{m}(\mtrx{q})+\varepsilon\mtrx{F}_\text{m}(\gmtrx\phi).\label{eq:ds_diag}
\end{gather}
We now consider the two-dimensional modal subspace $\mathcal{E}=\text{span}\left\{ \vec{v}_{1},\vec{v}_{2}\right\} \subset\mathbb{C}^{2n}$ with $\vec{v}_{2}=\bar{\vec{v}}_{1}.$ The remaining linearly independent eigenvectors $\vec{v}_{3},\ldots,\vec{v}_{2n}$ span a complex subspace $\mathcal{C}\subset\mathbb{C}^{2n}$ such that the full phase space of (\ref{eq:ds_diag}) can be expressed as the direct sum
\begin{equation}
\mathbb{C}^{2n}=\mathcal{E}\oplus\mathcal{C}.
\end{equation}
We write the diagonal matrix $\vec{\Lambda}$ as
\begin{equation}
\vec{\Lambda}=\left[\begin{array}{cc}
\vec{\Lambda}_{\mathcal{E}} & 0\\
0 & \vec{\Lambda}_{\mathcal{C}}
\end{array}\right],\quad\text{Spect}\left(\vec{\Lambda}_{\mathcal{E}}\right)=\left\{ \lambda_{1},\lambda_{2}\right\} ,\quad\text{Spect}\left(\vec{\Lambda}_{\mathcal{C}}\right)=\left\{ \lambda_{3},\ldots,\lambda_{2n}\right\} ,\label{eq:lin_decomp}
\end{equation}
with $\vec{\Lambda}_{\mathcal{E}}=\text{diag}(\lambda_{1},\lambda_{2})$
and $\vec{\Lambda}_{\mathcal{C}}=\text{diag}(\lambda_{3},\ldots,\lambda_{2n})$.

Following Haller and Ponsioen \cite{Haller2016}, we now define a \textit{non-autonomous spectral submanifold} (SSM), $\mathcal{W}(\mathcal{E})$, corresponding to the spectral subspace $\mathcal{E}$ of $\vec{\Lambda}$ as a two-dimensional invariant manifold of the dynamical system (\ref{eq:ds_diag}) that:
\begin{itemize}
\item [{(i)}] Perturbs smoothly from $\mathcal{E}$ at the trivial fixed point $\mtrx{q}=\mtrx{0}$ under the addition of the $\mathcal{O}(\varepsilon)$ terms in Eq. (\ref{eq:ds_diag}). 
\item [{(ii)}] Is strictly smoother than any other invariant manifold with the same properties.
\end{itemize}
We also define the \textit{absolute spectral quotient} $\Sigma(\mathcal{E})$ of $\mathcal{E}$ as the positive integer 
\begin{equation}
\Sigma(\mathcal{E})=\text{Int}\left[\frac{\min_{\lambda\in\text{Spect}(\vec{\Lambda})}\text{Re}\lambda}{\max_{\lambda\in\text{Spect}(\vec{\Lambda}_{\mathcal{E}})}\text{Re}\lambda}\right]\in\mathbb{N}^{+}.\label{eq:abs_spect_quo}
\end{equation}
Additionally, we introduce the non-resonance conditions 
\begin{equation}
a\text{Re}\lambda_{1}+b\text{Re}\lambda_{2}\neq\text{Re}\lambda_{l},\quad\forall\lambda_{l}\in\text{Spect}(\vec{\Lambda}_{\mathcal{C}}),\quad 2\leq a+b\leq\Sigma(\mathcal{E}),\quad a,b\in\mathbb{N}_0. \label{eq:ext_res}
\end{equation}
We now restate the following result from Haller and Ponsioen \cite{Haller2016}, for the existence of an SMM in system (\ref{eq:ds_diag}).
\begin{theorem}\label{thrm:SSM} 
Under the non-resonance conditions (\ref{eq:ext_res}), the following hold for system (\ref{eq:ds_diag}):
\begin{itemize}
\item [{(i)}] There exists a unique two-dimensional, time-periodic, analytic SSM, $\mathcal{W}(\mathcal{E})$.
\item [{(ii)}]$\mathcal{W}(\mathcal{E})$ can be viewed as an embedding of an open set $\mathcal{U}$ into the phase space of system (\ref{eq:ds_diag}) via the map
\begin{equation}
\vec{W}(\mtrx{s},\phi):\mathcal{U}\subset\mathbb{C}^{2}\times S^1\rightarrow\mathbb{C}^{2n}.\label{eq:W_map}
\end{equation}
We can approximate $\mtrx{W}(\mtrx{s},\phi)$ in a neighborhood of the origin using a Taylor expansion in the parameterization coordinates $\mtrx{s}$, having coefficients that depend periodically on the phase variable $\phi$.
\item [{(iii)}] There exists a polynomial function $\vec{R}(\mtrx{s},\phi):\mathcal{U}\rightarrow\mathcal{U}$
satisfying the invariance relationship 
\upshape
\begin{equation}
\mtrx{\Lambda} \mtrx{W}(\mtrx{s},\phi)+\mtrx{G}_\text{m}(\mtrx{W}(\mtrx{s},\phi))+\varepsilon\mtrx{F}_\text{m}(\phi)=D_\mtrx{s}\mtrx{W}(\mtrx{s},\phi)\mtrx{R}(\mtrx{s},\phi)+D_{\phi}\mtrx{W}(\mtrx{s},\phi)\Omega,\label{eq:invar}
\end{equation}
\itshape
such that the reduced dynamics on the SSM can be expressed as
\begin{equation}
\dot{\mtrx{s}}=\mtrx{R}(\mtrx{s},\phi). \label{eq:map_R}
\end{equation}
\end{itemize}
\end{theorem}
\begin{proof}:
We have simply restated the main theorem by Haller and Ponsioen \cite{Haller2016}, which is based on the more abstract results of Cabré et al. \cite{Cabre2003,Cabre2003b,Cabre2005} for mappings on Banach spaces. \qed
\end{proof}
We give an illustration of a time-periodic SSM in Fig. \ref{fig:SSM_illu}. We have assumed a case in which the SSM has three limit cycles for a given forcing frequency, with two of these limit cycles contained in an isola. The SSM approach can be viewed as a refinement and extension of the seminal work of Shaw and Pierre \cite{Shaw1993}, who envision nonlinear normal modes as invariant manifolds that are locally graphs over two-dimensional modal subspaces of the linearized system. We explain in detail how to construct time-periodic SSMs in Appendix \ref{app:coeff_eq}. 
\begin{figure}[H]
\centering
    \begin{subfigure}{1\textwidth}
        \centering
         \if\mycmd1
         \includegraphics[scale=0.5]{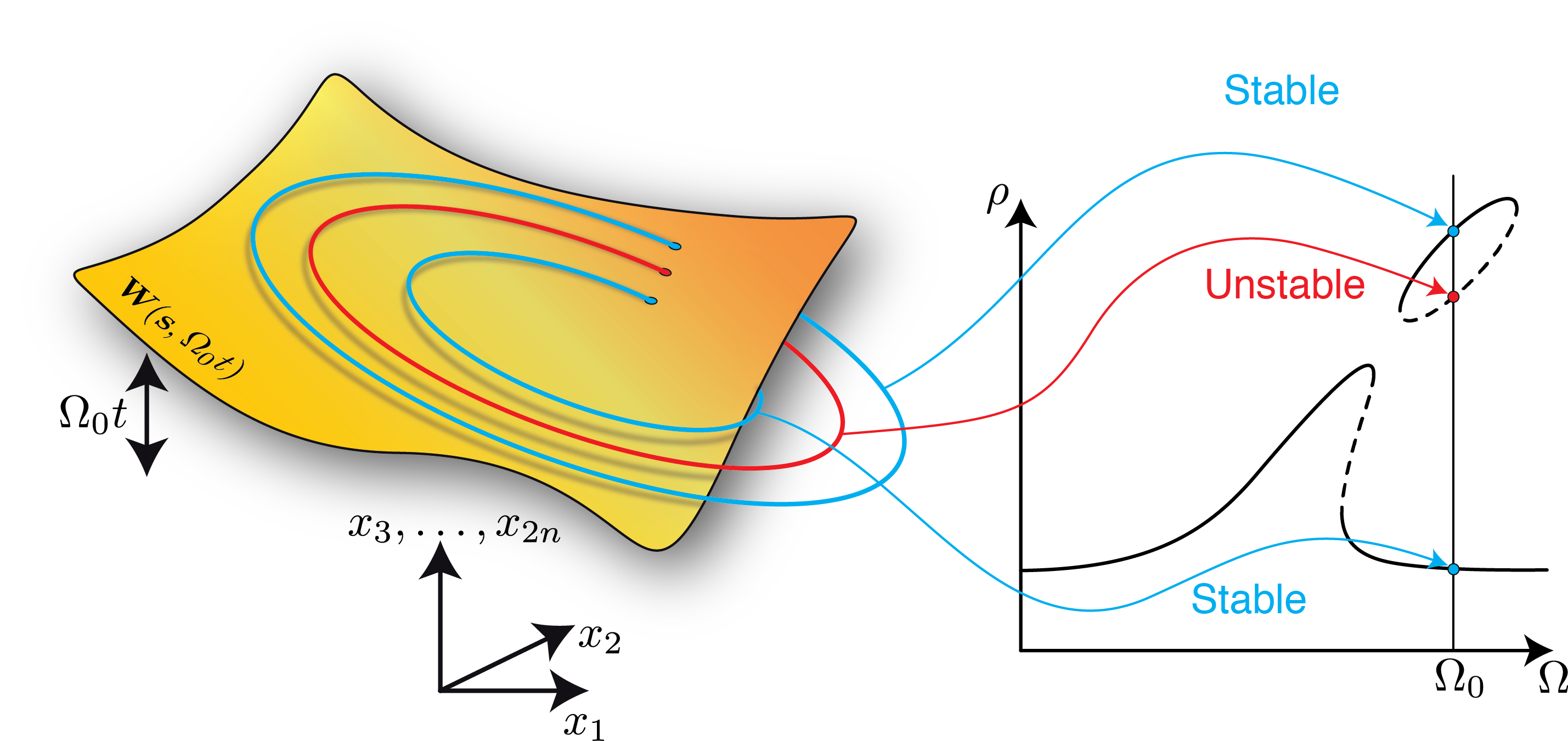}
         \fi
    \end{subfigure}
    \caption{Illustration of a time-periodic SSM. For a given forcing frequency $\Omega$, we illustrate how the SSM may contain three limit cycles, of which two fall in an isola. \label{fig:SSM_illu}}
\end{figure}
Our next result concerns the dynamics on the SSM described in Theorem \ref{thrm:SSM}
\begin{theorem}\label{thm:red_dyn}
The dynamics on the two-dimensional SSM given in Theorem \ref{thrm:SSM} can approximately be written in polar coordinates $(\rho,\psi)$ as
\begin{align}
\dot{\rho}&=a(\rho)+\varepsilon\left(f_1(\rho,\Omega)\cos(\psi)+f_2(\rho,\Omega)\sin(\psi)\right),\label{eq:red1_orig} \\
\dot{\psi}&=(b(\rho)-\Omega)+\frac{\varepsilon}{\rho}\left(g_1(\rho,\Omega)\cos(\psi)-g_2(\rho,\Omega)\sin(\psi)\right),\label{eq:red2_orig}
\end{align}
where
\upshape
\begin{align}
a(\rho)   &= \text{Re}(\lambda_1)\rho+\sum_{i=1}^M\text{Re}(\gamma_i)\rho^{2i+1},\label{eq:auto_a} \\
b(\rho)   &= \text{Im}(\lambda_1)+\sum_{i=1}^M\text{Im}(\gamma_i)\rho^{2i},  \\
f_1(\rho,\Omega) &= \text{Re}(c_{1,\mtrx{0}}) + \sum_{i=1}^M\left(\text{Re}(c_{1,(i,i)}(\Omega))+\text{Re}(d_{1,(i+1,i-1)}(\Omega))\right)\rho^{2i}, \\
f_2(\rho,\Omega) &= \text{Im}(c_{1,\mtrx{0}}) + \sum_{i=1}^M\left(\text{Im}(c_{1,(i,i)}(\Omega))-\text{Im}(d_{1,(i+1,i-1)}(\Omega))\right)\rho^{2i}, \\
g_1(\rho,\Omega) &= \text{Im}(c_{1,\mtrx{0}}) + \sum_{i=1}^M\left(\text{Im}(c_{1,(i,i)}(\Omega))+\text{Im}(d_{1,(i+1,i-1)}(\Omega))\right)\rho^{2i}, \\
g_2(\rho,\Omega) &= \text{Re}(c_{1,\mtrx{0}}) + \sum_{i=1}^M\left(\text{Re}(c_{1,(i,i)}(\Omega))-\text{Re}(d_{1,(i+1,i-1)}(\Omega))\right)\rho^{2i}, \label{eq:g2}
\end{align}
\itshape
with $2M+1$ denoting the order of the expansion. 
\end{theorem}
\begin{proof}:
We derive this result in Appendix \ref{app:red_dyn}. \qed
\end{proof}

In the unforced limit ($\varepsilon=0$), the reduced system (\ref{eq:red1_orig})-(\ref{eq:red2_orig}) can have fixed points but no nontrivial periodic orbits. This is because (\ref{eq:red1_orig}) decouples from (\ref{eq:red2_orig}), representing a one-dimensional ordinary differential equation that cannot have non-constant periodic solutions. By construction, the trivial fixed point of (\ref{eq:red1_orig})-(\ref{eq:red2_orig}) is asymptotically stable and will persist for $\varepsilon>0$. These persisting fixed points satisfy the system of equations
\begin{equation}
\mtrx{F}(\mtrx{u}) = 
\begin{bmatrix}
F_1(\mtrx{u})\\
F_2(\mtrx{u})
\end{bmatrix} = 
\begin{bmatrix}
a(\rho) + \varepsilon\left(f_1(\rho,\Omega)\cos(\psi)+f_2(\rho,\Omega)\sin(\psi)\right) \\
(b(\rho)-\Omega)\rho + \varepsilon\left(g_1(\rho,\Omega)\cos(\psi)-g_2(\rho,\Omega)\sin(\psi)\right) 
\end{bmatrix} = \mtrx{0}, \label{eq:zeroproblem}\\
\end{equation}
where
\begin{equation}
\mtrx{F}(\mtrx{u}):\mathbb{R}^3 \rightarrow \mathbb{R}^2, \quad
\mtrx{u} = 
\begin{bmatrix}
\rho \\
\Omega\\ 
\psi
\end{bmatrix}. \nonumber
\end{equation}

As we show in Appendix \ref{app:ext_fr}, under appropriate non-degeneracy conditions, the zeros of (\ref{eq:zeroproblem}) form a one-dimensional manifold, which, after a projection onto the amplitude-frequency space, will represent the FRC. The stability of these fixed points (which correspond to periodic solutions of the full mechanical system) is determined by the real parts of the eigenvalues of the Jacobian of $\mtrx{F}(\mtrx{u})$. 

\begin{theorem}\label{thm:gen_FR}The amplitude $\rho$ of the $T$-periodic orbits of the reduced dynamics (\ref{eq:red1_orig})-(\ref{eq:red2_orig}) are given by the zeros of the function 
\begin{equation}
G(\rho;\Omega)=(b(\rho)-\Omega)\rho + \varepsilon\left(g_1(\rho,\Omega)\frac{1-K^\pm(\rho;\Omega)^2}{1+K^\pm(\rho;\Omega)^2}-g_2(\rho,\Omega)\frac{2K^\pm(\rho;\Omega)}{1+K^\pm(\rho;\Omega)^2}\right)=0, \label{eq:G_thrm}
\end{equation}
where
\begin{equation}
K^\pm(\rho;\Omega)=\frac{-\varepsilon f_2(\rho,\Omega)\pm \sqrt{\varepsilon^2\left(f_1(\rho,\Omega)^2+f_2(\rho,\Omega)^2\right)-a(\rho)^2}}{a(\rho) - \varepsilon f_1(\rho,\Omega)}.\label{eq:K_thrm}
\end{equation}
\end{theorem}
\begin{proof}
We derive this result in Appendix \ref{app:impl_eq}.\qed
\end{proof}
The zero-level set of Eq. (\ref{eq:G_thrm}) yields the forced-response curve in the $(\Omega,\rho)$-space. This curve will consist of two segments obtained from $K^+$ and $K^-$ in Eq. (\ref{eq:K_thrm}). The two segments meet exactly at the point where the square root term in the definition of $K^\pm(\rho;\Omega)$ is equal to zero. We sketch this for a damped, nonlinear, periodically forced mechanical system in Fig. \ref{fig:FRC_intersection_impl}. 

\begin{figure}
\centering
    \begin{subfigure}{1\textwidth}
        \centering
         \if\mycmd1
         \includegraphics[scale=0.3]{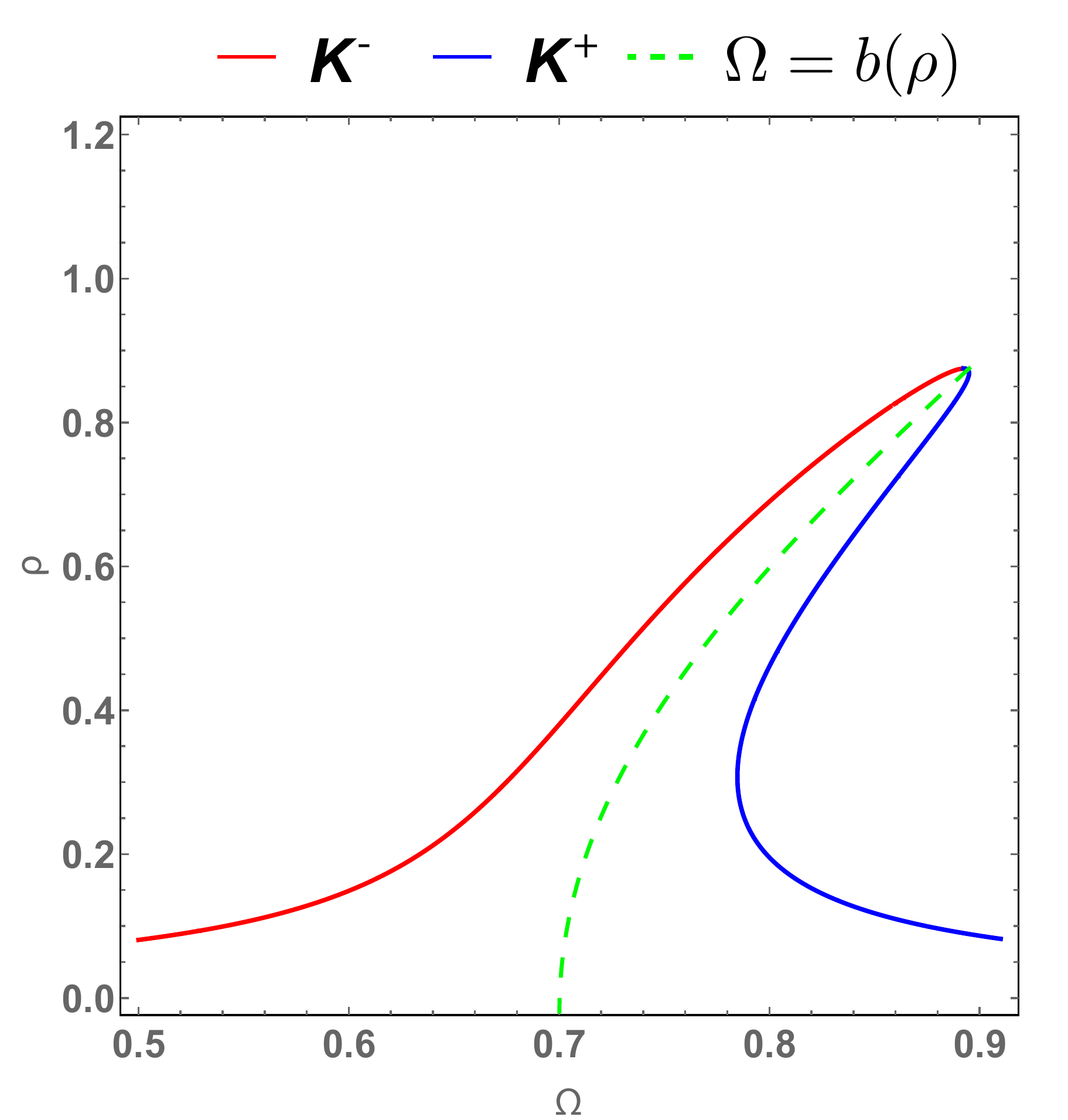}
         \fi
    \end{subfigure}
    \caption{Example of the zero-level set of Eq. (\ref{eq:G_thrm}) for a damped, non-linear, periodically forced mechanical system  with a hardening nonlinearity. The blue and red curves correspond to $K^+$ and $K^-$. These two segments come together exactly at the point where the discriminant of the quadratic Eq. (\ref{eq:quadK}) is equal to zero.  \label{fig:FRC_intersection_impl}}
\end{figure}
Because the isolated branches of periodic solutions are also a part of the FRC, the zero-level set of $G(\rho;\Omega)$ can predict isolas as well. In contrast, detecting isolas by numerical continuation requires one to start on the isola and hence assumes a priori knowledge of an isolated branch of periodic solutions. We will discuss this in section \ref{sec:isola}.

\section{Analytic criterion for isolas \label{sec:isola}}
We will now give an analytic criterion for the emergence of isolas in terms of the function $a(\rho)$ defined in equation (\ref{eq:auto_a}). An essential question is how approximate zeros obtained from finite-order Taylor series expansions persist as $M\rightarrow\infty$. Jentzsch \cite{jentzsch1916untersuchungen} proved that in the limit of the order of the Taylor series expansion going to infinity, the non-persistent spurious zeros come arbitrarily close to the boundary of the domain of convergence. Hurwitz \cite{hurwitz1888ueber} showed that in the same limit, the uniform convergence of the Taylor series polynomial leads to a good approximation of the genuine zeros inside the circle of convergence. Christiansen and Madsen\cite{Christiansen2006} numerically verified this behavior on several examples. We will call such a zero $\rho_0$ of $a(\rho)$, non-spurious if it converges to a genuine zero of $a(\rho)$ in the limit of $M\rightarrow\infty$.

\begin{theorem}\label{thrm:birth_isola}
Assume that $\rho_0\neq 0$ is a non-spurious transverse zero of $a(\rho)$, i.e.,
\begin{equation}
a(\rho_0) = 0, \quad \partial_\rho a(\rho_0) \neq 0.
\end{equation} 
Then, for $\varepsilon>0$ small enough, system (\ref{eq:zeroproblem}) has an isola that perturbs from the unforced damped backbone curve $\Omega=b(\rho)$ near the amplitude value $\rho_0$. 
\end{theorem}
\begin{proof}
We derive this result in Appendix \ref{app:birth_isola}. \qed
\end{proof}

In order to verify if a non-trivial zero of the Taylor series expansion of $a(\rho)$ is also non-spurious, we compute the (generally complex) zeros of the function $a(\rho)$ for increasing order of approximation $M$. As we discussed before Theorem \ref{thrm:birth_isola}, spurious zeros will converge to the circle defining the radius of convergence of $a(\rho)$, whereas non-spurious zeros stay bounded away from that circle and hence converge to the genuine zeros of $a(\rho)$.

In Fig. \ref{fig:inter_perb}, we sketch qualitatively the statement of Theorem  \ref{thrm:birth_isola}: a non-spurious, transverse zero $\rho_0$ of $a(\rho)$ indicates a nearby isola.
\begin{figure}
\centering
    \begin{subfigure}{0.45\textwidth}
        \centering
        \if\mycmd1
        \includegraphics[scale=0.6]{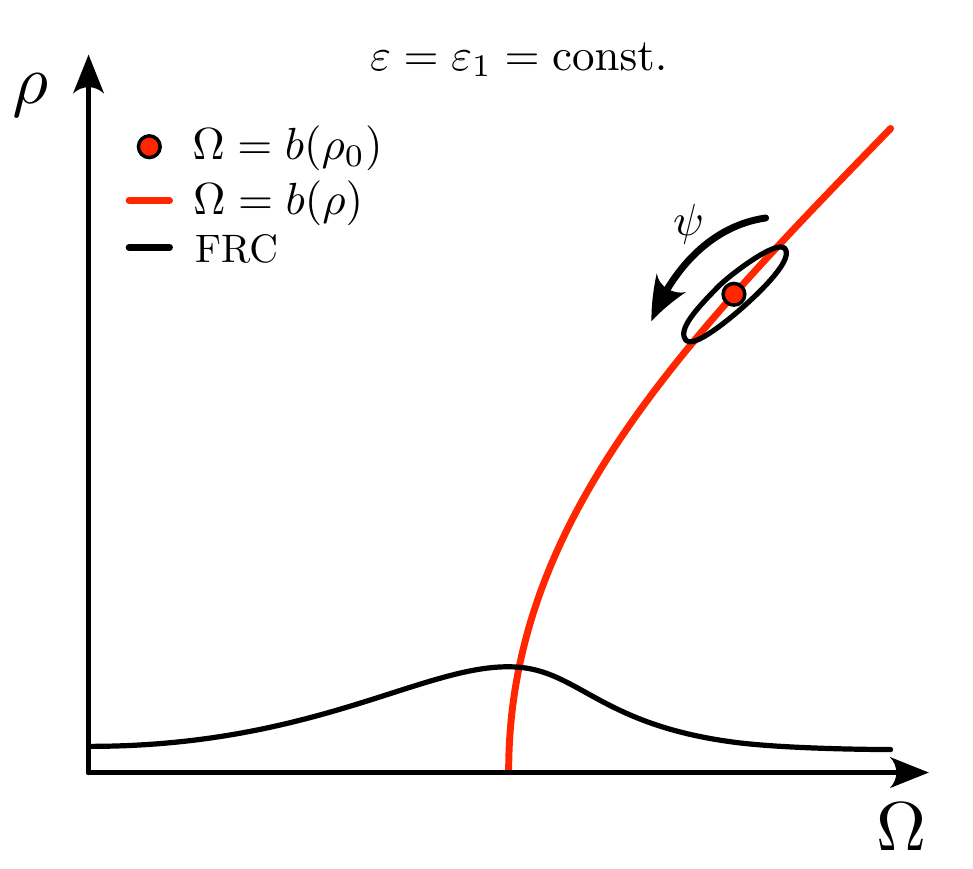}
        \fi
    \end{subfigure}
    \begin{subfigure}{0.45\textwidth}
        \centering
        \if\mycmd1
        \includegraphics[scale=0.6]{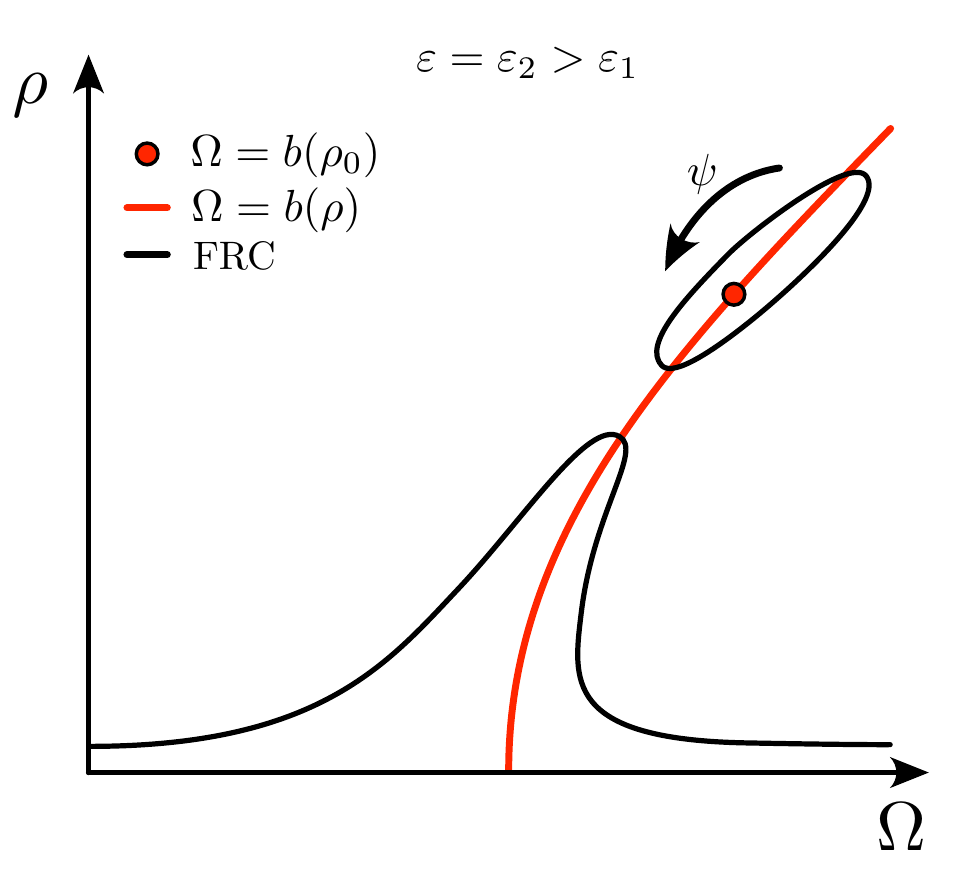}
        \fi
    \end{subfigure}
    \caption{An isola of periodic solutions near the damped backbone curve, emerging from a non-trivial transverse zero $\rho_0$ of $a(\rho)$. The isola curve can be parameterized by the variable $\psi$.  \label{fig:inter_perb}}
\end{figure}

\subsection{Leading-order analytic formula for isolas}
For higher-order approximations of $a(\rho)$, we can determine the roots of $a(\rho)$ numerically. Restricting ourselves to a $3^\text{rd}$-order approximation of the SSM, we can, however, extract even an analytic criterion for the existence of an isola. Following the work of Breunung and Haller \cite{Breunung2017}, we truncate the parameterization $\mtrx{W}(\mtrx{s},\phi)$ and the reduced dynamics $\mtrx{R}(\mtrx{s},\phi)$ at $\mathcal{O}(\varepsilon|\mtrx{s}|,\varepsilon^2)$, which they justify by introducing the scaling $\mtrx{s}\rightarrow \varepsilon^{\frac{1}{4}}\mtrx{s}$, such that the zero problem of the reduced system can be written as


\begin{align}
\tilde{\mtrx{F}}(\mtrx{u}) = 
\begin{bmatrix}
\underbrace{\text{Re}(\lambda_1)\rho + \text{Re}(\gamma_1)\rho^3}_{a(\rho)} + \varepsilon\left(\text{Re}(c_{1,\mtrx{0}})\cos(\psi)+\text{Im}(c_{1,\mtrx{0}})\sin(\psi)\right) \\
(\underbrace{\text{Im}(\lambda_1)+ \text{Im}(\gamma_1)\rho^2}_{b(\rho)}-\Omega)\rho  + \varepsilon\left(\text{Im}(c_{1,\mtrx{0}})\cos(\psi)-\text{Re}(c_{1,\mtrx{0}})\sin(\psi)\right) 
\end{bmatrix} = \mtrx{0}. \label{eq:zeroredsys}
\end{align}
Here we have $f_1 = g_2= \text{Re}(c_{1,\mtrx{0}})$ and $f_2 = g_1 = \text{Im}(c_{1,\mtrx{0}})$. We now show in Theorem \ref{thrm:isola} that this approximation gives an analytically computable condition for the existence of an isola.  

\begin{theorem}\label{thrm:isola}
Assume that for system (\ref{eq:zeroredsys}), $\textup{Re}(\gamma_1)>0$ is satisfied and the cubic-order zero $\rho_1=\sqrt{|\textup{Re}(\lambda_1)|/\textup{Re}(\gamma_1)}$ of $a(\rho)$ is non-spurious. Then the following holds:
\begin{itemize}
\item[{(i)}]For $\varepsilon>0$ small enough, an isola of the type described in Theorem \ref{thrm:birth_isola} exists near the point $(\Omega,\rho)=(b(\rho_0),\rho_0)$ of the damped backbone curve.
\item [{(ii)}] The isola will be disconnected from the main FRC for $\varepsilon>0$ values satisfying
\begin{equation}
\varepsilon < \frac{1}{\norm{c_{1,\mtrx{0}}}}\sqrt{\frac{4|\textup{Re}(\lambda_1)|^3}{27 \textup{Re}(\gamma_1)}}.
\end{equation}
\item [{(iii)}] The isola will merge with the main FRC approximately at the $\varepsilon$ value
\begin{equation}
\varepsilon_\textup{m} =  \frac{1}{\norm{c_{1,\mtrx{0}}}}\sqrt{\frac{4|\textup{Re}(\lambda_1)|^3}{27 \textup{Re}(\gamma_1)}}. \label{eq:isola_merge}
\end{equation}
\end{itemize}
\end{theorem}
\begin{proof}
We derive this result in Appendix \ref{app:isola}.\qed
\end{proof}
\section{Numerical Examples \label{sec:num_examples}}
\subsection{The modified Shaw\textendash Pierre example}
As a typical benchmark, we first consider a modified version of the example of Shaw and Pierre \cite{Shaw1994}, in which an additional cubic nonlinear damper is added, as in the single degree-of-freedom example of Habib et al.\ \cite{Habib2017}. The equations of motion of our two-degree-of-freedom system in first-order form are given by
\begin{equation}
\dot{\mtrx{x}} =
\underbrace{
\begin{bmatrix}[c]
0 & 0 & 1 & 0 \vspace{1mm}\\
0 & 0 & 0 & 1 \vspace{1mm} \\
-\dfrac{2k}{m} & \dfrac{k}{m} & -\dfrac{c_1+c_2}{m} & \dfrac{c_2}{m} \vspace{1mm} \\
\dfrac{k}{m} & -\dfrac{2k}{m} & \dfrac{c_2}{m} & -\dfrac{c_1+c_2}{m}  \vspace{1mm}
\end{bmatrix}}_{\mtrx{A}} \mtrx{x} +
\underbrace{
\begin{bmatrix}[c]
0   \vspace{1mm}\\ 
0   \vspace{1mm}\\
-\dfrac{\kappa}{m}x_1^3-\dfrac{\alpha}{m}x_3^3 \vspace{1mm}\\ 
0  \vspace{1mm}
\end{bmatrix}}_{\mtrx{G}_\text{p}(\mtrx{x})}+
\varepsilon 
\underbrace{
\begin{bmatrix}
0  \vspace{1mm}\\
0  \vspace{1mm}\\
\dfrac{P}{m}\cos(\Omega t)  \vspace{1mm}\\
0  \vspace{1mm}
\end{bmatrix}}_{\mtrx{F}_\text{p}(\Omega t)}, \label{eq:SP_eom}
\end{equation}
where $\mtrx{x}=[x_1,x_2,x_3,x_4]^\top=[y_1,y_2,\dot{y}_1,\dot{y}_2]^\top$. The matrix $\mtrx{A}$ has the eigenvalue pairs
\begin{align}
& \lambda_{1,2}=\left(-\zeta_1 \pm \mi \sqrt{1-\zeta_1^2}\right)\omega_1,\quad \zeta_1 = \frac{c_1}{2m\omega_1},\quad \omega_1 = \sqrt{\frac{k}{m}}, \\ 
& \lambda_{3,4}=\left(-\zeta_2 \pm \mi \sqrt{1-\zeta_2^2}\right)\omega_2,\quad \zeta_2 = \frac{c_1+2c_2}{2m\omega_2},\quad \omega_2 = \sqrt{\frac{3k}{m}},
\end{align}
assume that both modes are underdamped, i.e., $0<\zeta_1<1$ and $0<\zeta_2<1$. The matrix $\mtrx{T}$ that transforms our system to complex modal coordinates is composed of the eigenvectors of our system, i.e.,
\begin{equation}
\mtrx T = 
\begin{bmatrix}[r]
1 & 1 & 1 & 1    \\
1 & 1 & -1 & -1 \\
\lambda_1 & \bar{\lambda}_1 & \lambda_3 & \bar{\lambda}_3 \\
\lambda_1 & \bar{\lambda}_1 & -\lambda_3 & -\bar{\lambda}_3
\end{bmatrix},
\end{equation}
with the inverse of $\mtrx{T}$ given by
\begin{equation}
\mtrx{T}^{-1} = 
\begin{bmatrix}[r]
-\dfrac{\bar{\lambda}_1}{2(\lambda_1-\bar{\lambda}_1)} & 
-\dfrac{\bar{\lambda}_1}{2(\lambda_1-\bar{\lambda}_1)} &
\dfrac{1}{2(\lambda_1-\bar{\lambda}_1)} &
\dfrac{1}{2(\lambda_1-\bar{\lambda}_1)} \vspace{1mm} \\

\dfrac{\lambda_1}{2(\lambda_1-\bar{\lambda}_1)} & 
\dfrac{\lambda_1}{2(\lambda_1-\bar{\lambda}_1)} &
-\dfrac{1}{2(\lambda_1-\bar{\lambda}_1)} &
-\dfrac{1}{2(\lambda_1-\bar{\lambda}_1)} \vspace{1mm} \\

-\dfrac{\bar{\lambda}_3}{2(\lambda_3-\bar{\lambda}_3)} & 
\dfrac{\bar{\lambda}_3}{2(\lambda_3-\bar{\lambda}_3)} &
\dfrac{1}{2(\lambda_3-\bar{\lambda}_3)} &
-\dfrac{1}{2(\lambda_3-\bar{\lambda}_3)} \vspace{1mm} \\

\dfrac{\lambda_3}{2(\lambda_3-\bar{\lambda}_3)} & 
-\dfrac{\lambda_3}{2(\lambda_3-\bar{\lambda}_3)} &
-\dfrac{1}{2(\lambda_3-\bar{\lambda}_3)} &
\dfrac{1}{2(\lambda_3-\bar{\lambda}_3)} 
\end{bmatrix}.
\end{equation}

In this example, we can compute the cubic coefficient of $a(\rho)$ explicitly.  Specifically, we have

\begin{align}
\text{Re}(\gamma_1) &= \text{Re}\left(-\frac{3\alpha}{m}[\mtrx{T}^{-1}]_{1,3}[\mtrx{T}]_{3,1}^2 [\mtrx{T}]_{3,2} 
-\frac{3\kappa}{m}[\mtrx{T}^{-1}]_{1,3}[\mtrx{T}]_{1,1}^2 [\mtrx{T}]_{1,2}\right) \\
& =\text{Re}\left(- \frac{3\left(\alpha\lambda_1^2\bar{\lambda}_1+\kappa\right)}{2m(\lambda_1-\bar{\lambda}_1)}\right) \\
& =-\frac{3\alpha k}{4m^2}. \label{eq:cubic_SP}
\end{align}
Therefore, for $\alpha<0$, the reduced dynamics on the third-order autonomous SSM will have a nontrivial zero. If, additionally, this zero is non-spurious, then Theorem \ref{thrm:isola} guarantees the existence of an isola. Using Eq. (\ref{eq:isola_merge}), the isola will merge with the main FRC for
\begin{equation}
\varepsilon_\text{m} = \frac{8m\sqrt{1-\zeta_1^2}\omega_1}{|P|}\sqrt{\frac{16m^2(\zeta_1\omega_1)^3}{81 k |\alpha|}}. \label{eq:merge_SP}
\end{equation}
We verify this analytic prediction numerically in Example \ref{ex:example_SP_1} below.
\begin{example}\label{ex:example_SP_1}
We choose the parameter values listed in Table \ref{tab:system_par_SP_ex_1} and compute the forced response curve for system (\ref{eq:SP_eom}). Note that for this choice of damping parameters, the non-resonance conditions (\ref{eq:ext_res}) are satisfied.
\begin{table}[H]
\begin{centering}
\caption{Parameter values for Example \ref{ex:example_SP_1}. \label{tab:system_par_SP_ex_1}}
\begin{tabular}{|c|c|}
\hline 
Symbol & Value \tabularnewline
\hline 
\hline 
$m$ & $\unit[1]{kg}$\tabularnewline
\hline 
$c_1$ & $\unit[0.03]{N\cdot s/m}$ \tabularnewline
\hline
$c_2$ & $\unit[\sqrt{3}\cdot0.03]{N\cdot s/m}$ \tabularnewline
\hline  
$k$ & $\unit[3]{N/m}$ \tabularnewline
\hline 
$\kappa$ & $\unit[0.4]{N/m^3}$ \tabularnewline
\hline 
$\alpha$ & $\unit[-0.6]{N \cdot (s/m)^3}$  \tabularnewline
\hline 
$P$ & $\unit[3]{N}$ \tabularnewline
\hline
\end{tabular}
\par\end{centering}
\end{table}
Plugging in the parameter values of Table \ref{tab:system_par_SP_ex_1} into Eq. (\ref{eq:cubic_SP}), we observe that the third-order coefficient of the autonomous part of the SSM is
\begin{equation}
\text{Re}(\gamma_1)=-\frac{3\alpha k}{4m^2} = 1.35 > 0. \label{eq:ex1_zero}
\end{equation} 
We now numerically verify that this transverse zero is non-spurious by computing the (complex) roots of $a(\rho)$ for an increasing order $M$ of expansion in formula (\ref{eq:auto_a}) using \textsc{SSMtool}. In Fig. \ref{fig:FRC_SP_roots}, we show these roots, up to $50^\text{th}$ order, with lighter colors indicating higher orders of approximation. Eq. (\ref{eq:ex1_zero}) and Fig. \ref{fig:FRC_SP_roots} allow us to conclude from statement (i) of Theorem \ref{thrm:isola} the existence of an isola near the amplitude value $|\rho_1^\pm|$, where $\rho_1^\pm$ are the two nontrivial, non-spurious zeros of $a(\rho)$ seen in Fig. \ref{fig:FRC_SP_roots}.
\begin{figure}
\centering
    \begin{subfigure}{1\textwidth}
        \centering
        \if\mycmd1
        \includegraphics[scale=0.45]{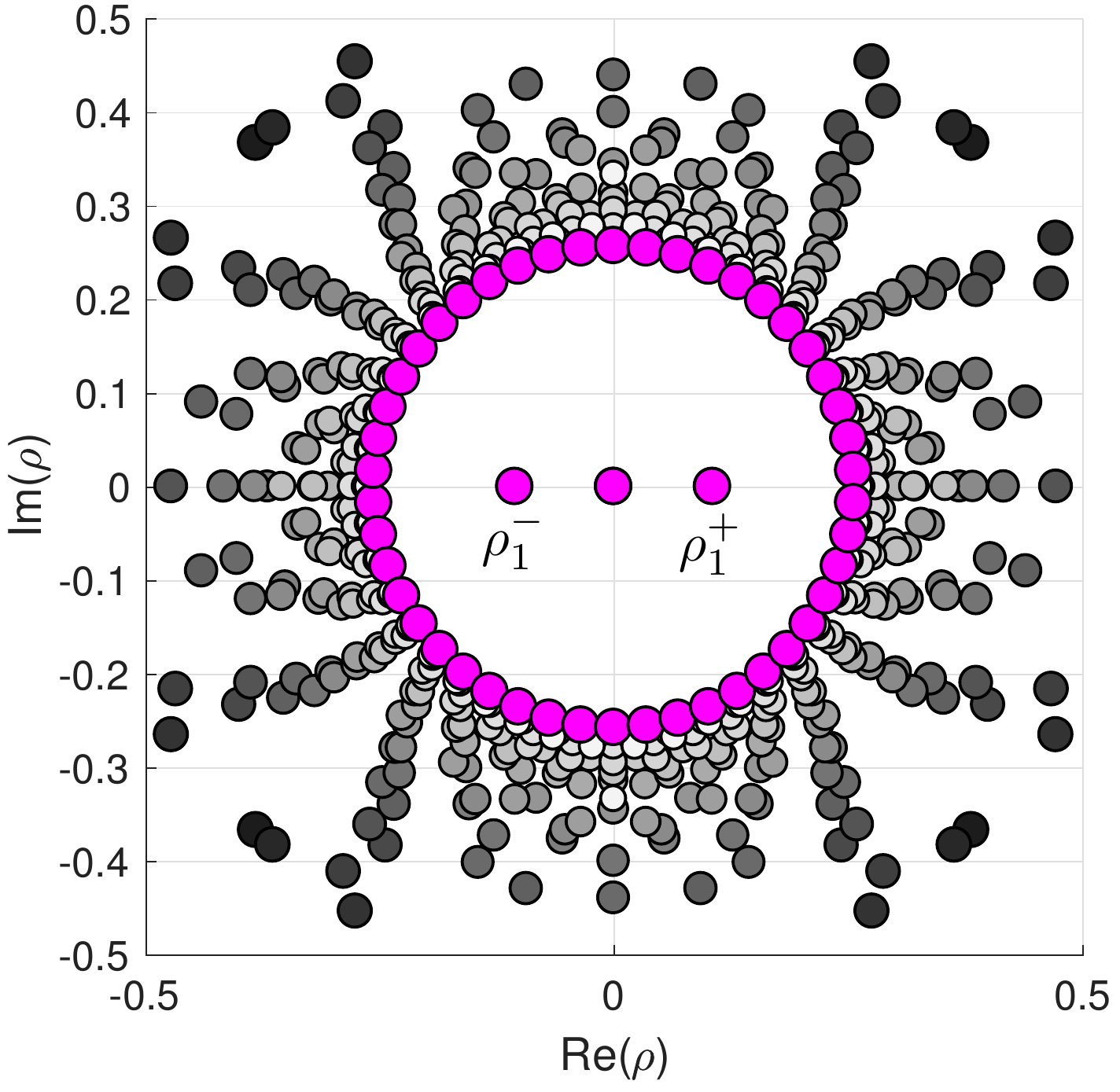}
        \fi
    \end{subfigure}
    \caption{Plot of the roots of $a(\rho)$ in the complex plane for Example \ref{ex:example_SP_1}, with brighter colors indicating an increasing order $M$ in the expansion of $a(\rho)$, up to 50th order (roots that are negative of each other are to be identified). The zeros from the highest approximation are highlighted in magenta. We observe that a non-trivial transverse zero $\rho_1^\pm$ persists for higher-order approximations and is clearly within the domain of convergence of the function $a(\rho)$.  \label{fig:FRC_SP_roots}}
\end{figure}
By Eq. (\ref{eq:merge_SP}), the isola will merge with the main branch of the forced response curve approximately at the parameter value
\begin{equation}
\varepsilon_\text{m}= 0.0028.
\end{equation} 
We now verify this analytic prediction for the isola merger numerically. In Fig. \ref{fig:FRC_SP}, we show in red the leading-order forced response curves for $\varepsilon=0.0027$ and $\varepsilon=0.0029$. Also shown in black are the forced response curves of the full system obtained via the periodic-orbit toolbox of $\textsc{coco}$ \cite{Dankowicz2013}. We conclude that the FRC obtained from our two-dimensional, SSM-reduced system perfectly predicts the behavior of the full system. 
\begin{figure}[h!]
\centering
    \begin{subfigure}{0.45\textwidth}
        \centering
        \if\mycmd1
        \includegraphics[scale=0.4]{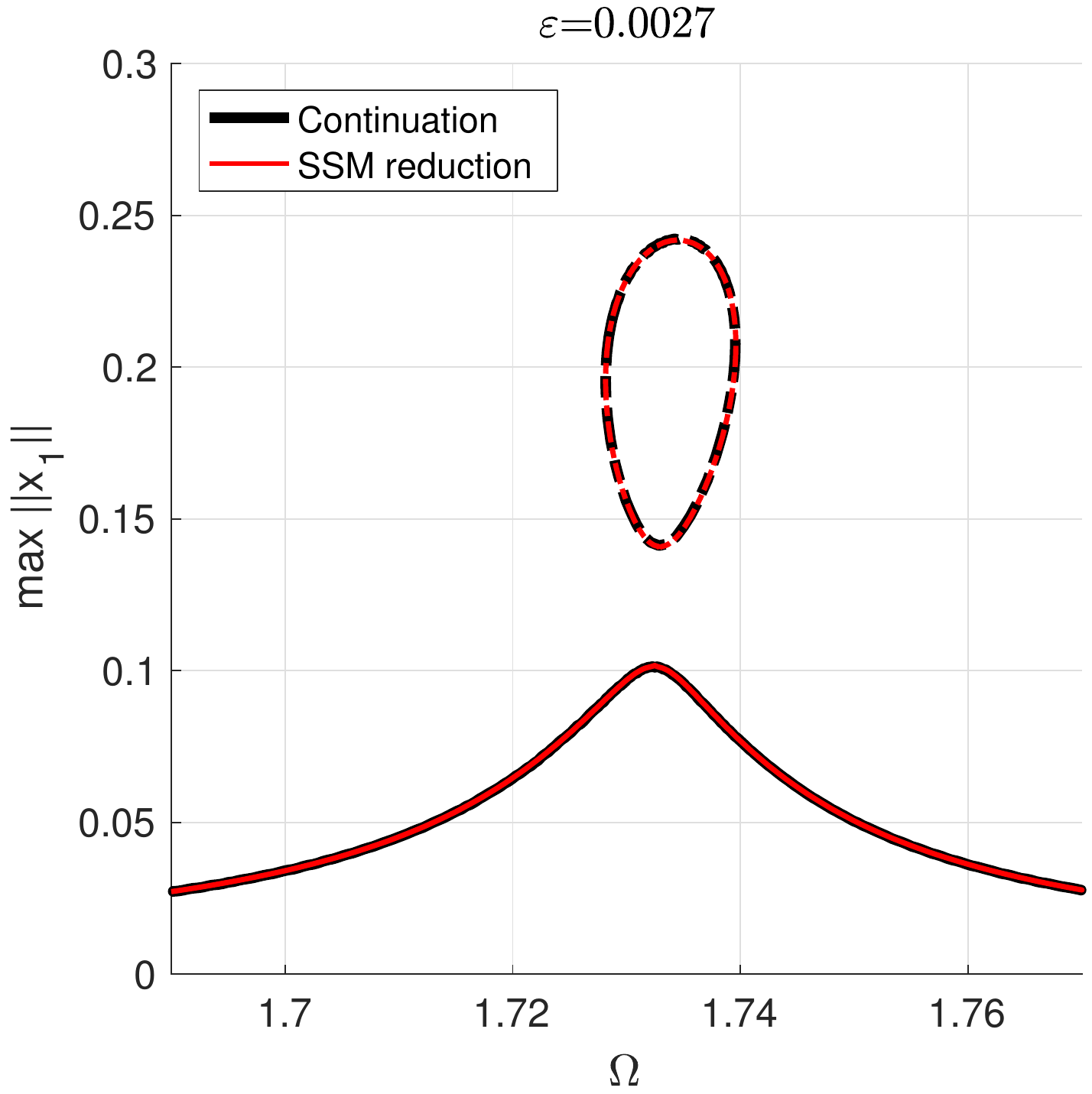}
        \fi
        \hspace*{8mm}
        \caption{\label{fig:FRC_SP_L}}
    \end{subfigure}
    \hspace*{5mm}
    \begin{subfigure}{0.45\textwidth}
        \centering
        \if\mycmd1
        \includegraphics[scale=0.4]{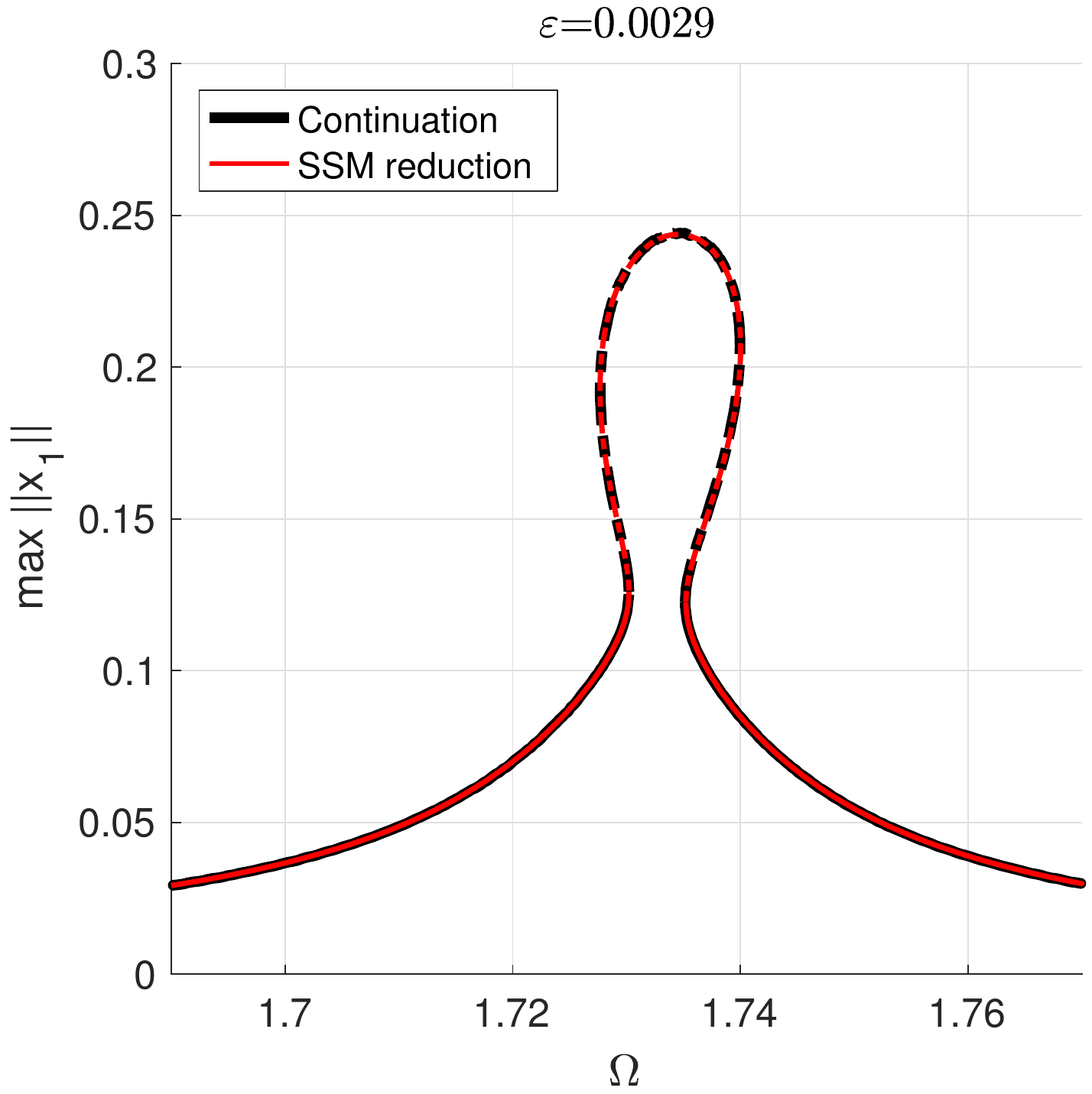}
        \fi
        \hspace*{8mm}
        \caption{\label{fig:FRC_SP_U}}
    \end{subfigure}
    \caption{(a) Forced response curve for $\varepsilon=0.0027$ in Example \ref{ex:example_SP_1}, which is slightly below the predicted value $\varepsilon_\text{m}$ for the merger of the isola with the main branch of the FRC. The dashed lines indicate that the isola is unstable. (b) Forced response curve for $\varepsilon=0.0029$, which is slightly above the predicted merging value $\varepsilon_m$. The unstable isola has indeed merged with the main FRC branch, as predicted analytically. \label{fig:FRC_SP}}
\end{figure}
\end{example}
\begin{example}\label{ex:example_SP_2}
In this example, we add a quintic nonlinear damper to system (\ref{eq:SP_eom}), which yields the modified equations of motion in first-order form 
\begin{equation}
\dot{\mtrx{x}} =
\underbrace{
\begin{bmatrix}[c]
0 & 0 & 1 & 0 \vspace{1mm}\\
0 & 0 & 0 & 1 \vspace{1mm} \\
-\dfrac{2k}{m} & \dfrac{k}{m} & -\dfrac{c_1+c_2}{m} & \dfrac{c_2}{m} \vspace{1mm} \\
\dfrac{k}{m} & -\dfrac{2k}{m} & \dfrac{c_2}{m} & -\dfrac{c_1+c_2}{m}  \vspace{1mm}
\end{bmatrix}}_{\mtrx{A}} \mtrx{x} +
\underbrace{
\begin{bmatrix}[c]
0   \vspace{1mm}\\ 
0   \vspace{1mm}\\
-\dfrac{\kappa}{m}x_1^3-\dfrac{\alpha}{m}x_3^3 -\dfrac{\beta}{m}x_3^5 \vspace{1mm}\\ 
0  \vspace{1mm}
\end{bmatrix}}_{\mtrx{G}_\text{p}(\mtrx{x})}+
\varepsilon 
\underbrace{
\begin{bmatrix}
0  \vspace{1mm}\\
0  \vspace{1mm}\\
\dfrac{P}{m}\cos(\Omega t)  \vspace{1mm}\\
0  \vspace{1mm}
\end{bmatrix}}_{\mtrx{F}_\text{p}(\Omega t)}.
\end{equation}
We again use the parameter values in Table \ref{tab:system_par_SP_ex_1} and additionally select the quintic damping coefficient $\beta=\unit[1.2]{N \cdot (s/m)^5}$. We use \textsc{SSMtool} to calculate the functions included in the reduced dynamics (\ref{eq:red1_orig})-(\ref{eq:red2_orig}) up to $5^\text{th}$ order in $\rho$. The function $a(\rho)$ now has two positive, non-spurious zeros located at $\rho^+_1=0.13$ and $\rho^+_2=0.17$. Therefore, Theorem \ref{thrm:birth_isola} implies the existence of two separate isolas bifurcating from the damped backbone curve under periodic forcing.
\begin{figure}
\centering
    \begin{subfigure}{1\textwidth}
        \centering
		\if\mycmd1     
        \includegraphics[scale=0.45]{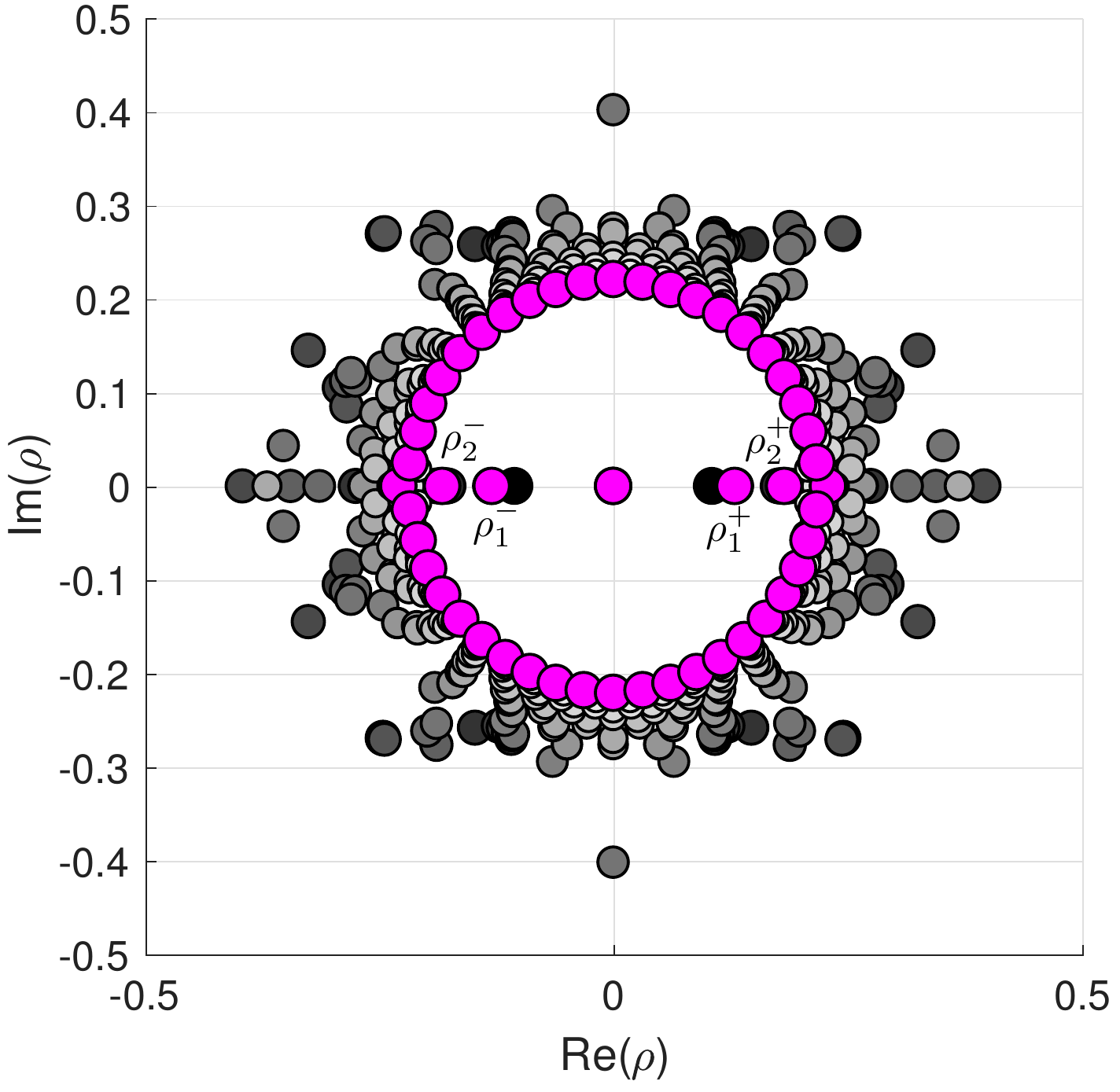}
        \fi
    \end{subfigure}
    \caption{Plot of the roots of $a(\rho)$ in the complex plane for Example \ref{ex:example_SP_2}, with brighter colors indicating an increasing order $M$ in the expansion of $a(\rho)$, up to 50th order (roots that are negative of each other are to be identified). The zeros from the highest approximation are highlighted in magenta. We observe that the non-trivial transverse zeros $\rho_1^\pm$ and $\rho_2^\pm$ persist for higher-order approximations and is clearly within the domain of convergence of the function $a(\rho)$.\label{fig:FRC_SP_O5_roots}}
\end{figure}

We show the extracted forced response curves for three different values of $\varepsilon$ in Fig. \ref{fig:FRC_SP_O5}. As we have predicted above, two isolas are born out of the non-trivial transverse zeros of $a(\rho)$ along the autonomous backbone curve. The isola with lower amplitudes is unstable, whereas the isola with higher amplitudes is partially stable. If we increase the forcing amplitude $\varepsilon$, the two isolas merge. Increasing $\varepsilon$ further will make the merged isolas merge with the lower FRC branch. The branches of the forced response curve, extracted using SSM theory, are again verified using the periodic-orbit toolbox of $\textsc{coco}$ \cite{Dankowicz2013}. In order to initialize the continuation algorithm, we integrate the full system to provide an initial solution guess that is used to start the continuation process. For higher amplitude values, our $5^\text{th}$-order approximation slightly deviates from the numerical continuation results, as expected. For lower amplitudes, however, our SSM-based prediction perfectly matches the numerical result.

\begin{figure}
\centering
    \begin{subfigure}{0.45\textwidth}
        \centering
        \if\mycmd1
        \includegraphics[scale=0.39]{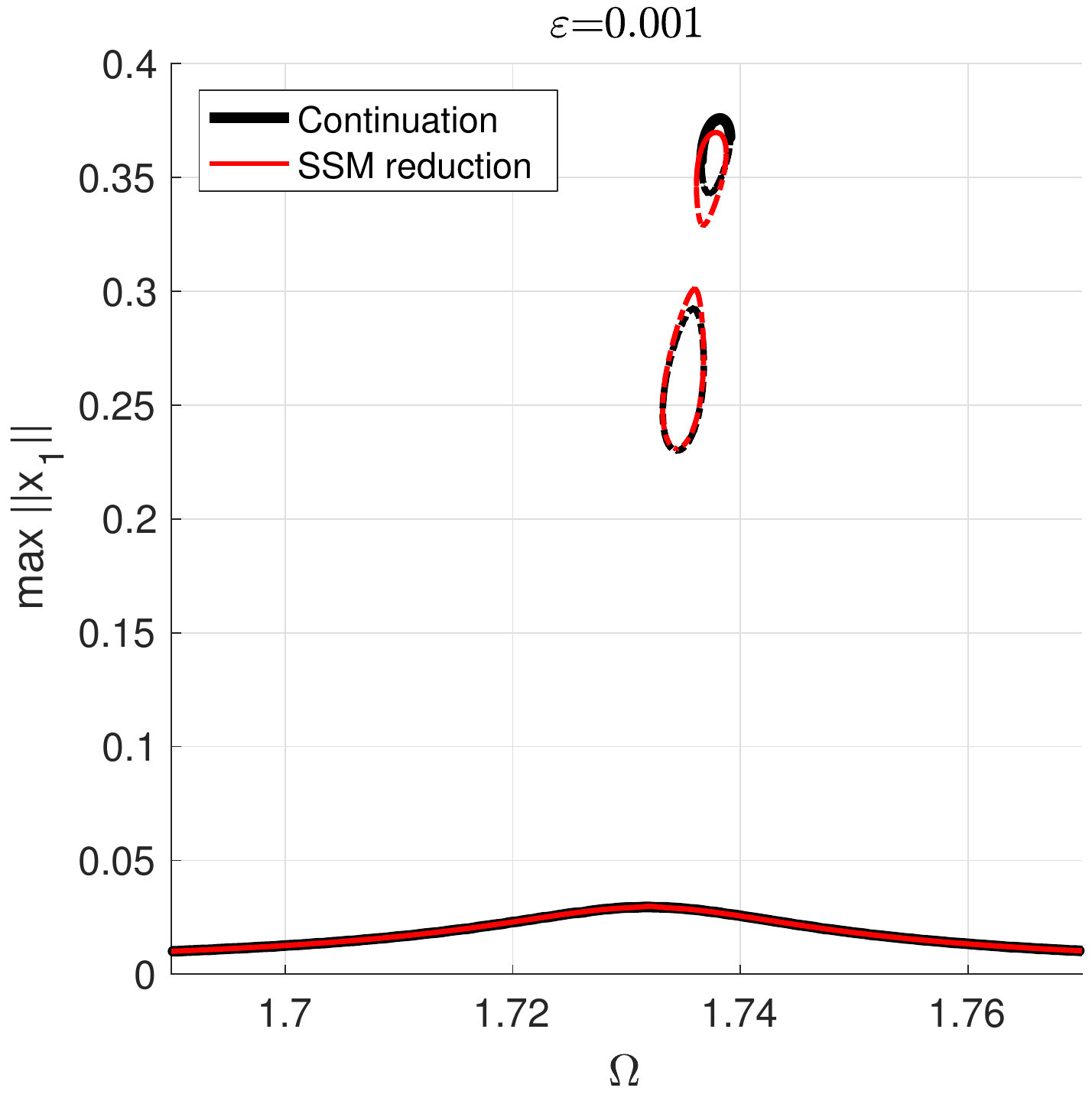}
        \fi
        \hspace*{8mm}
        \caption{\label{fig:FRC_SP_O5_1}}
        \vspace*{2mm}
    \end{subfigure}
    \hspace*{5mm}
    \begin{subfigure}{0.45\textwidth}
        \centering
        \if\mycmd1
        \includegraphics[scale=0.39]{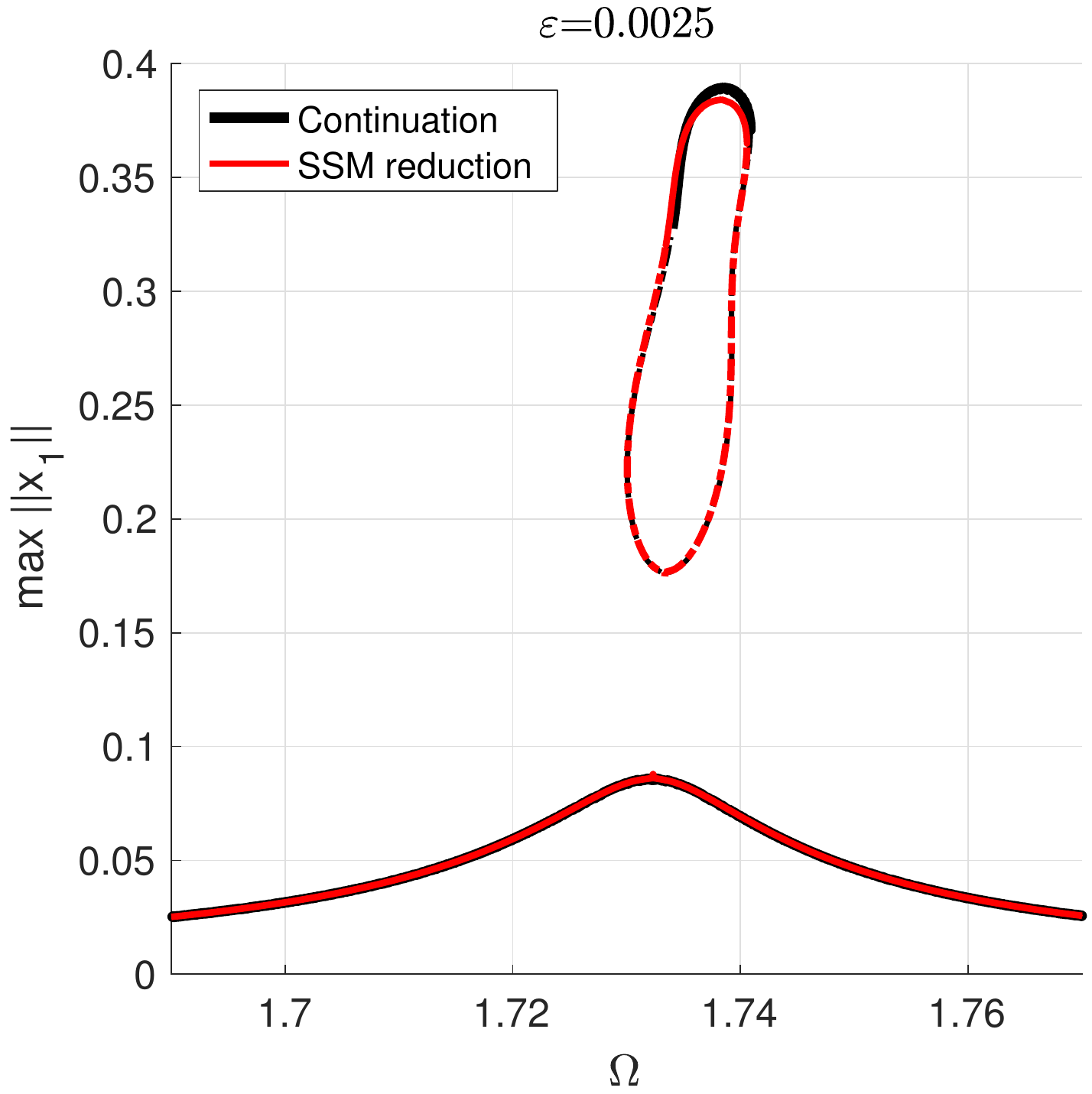}
        \fi
        \hspace*{8mm}
        \caption{\label{fig:FRC_SP_O5_2}}
        \vspace*{2mm}
    \end{subfigure}
    \begin{subfigure}{1\textwidth}
        \centering
        \if\mycmd1
        \includegraphics[scale=0.39]{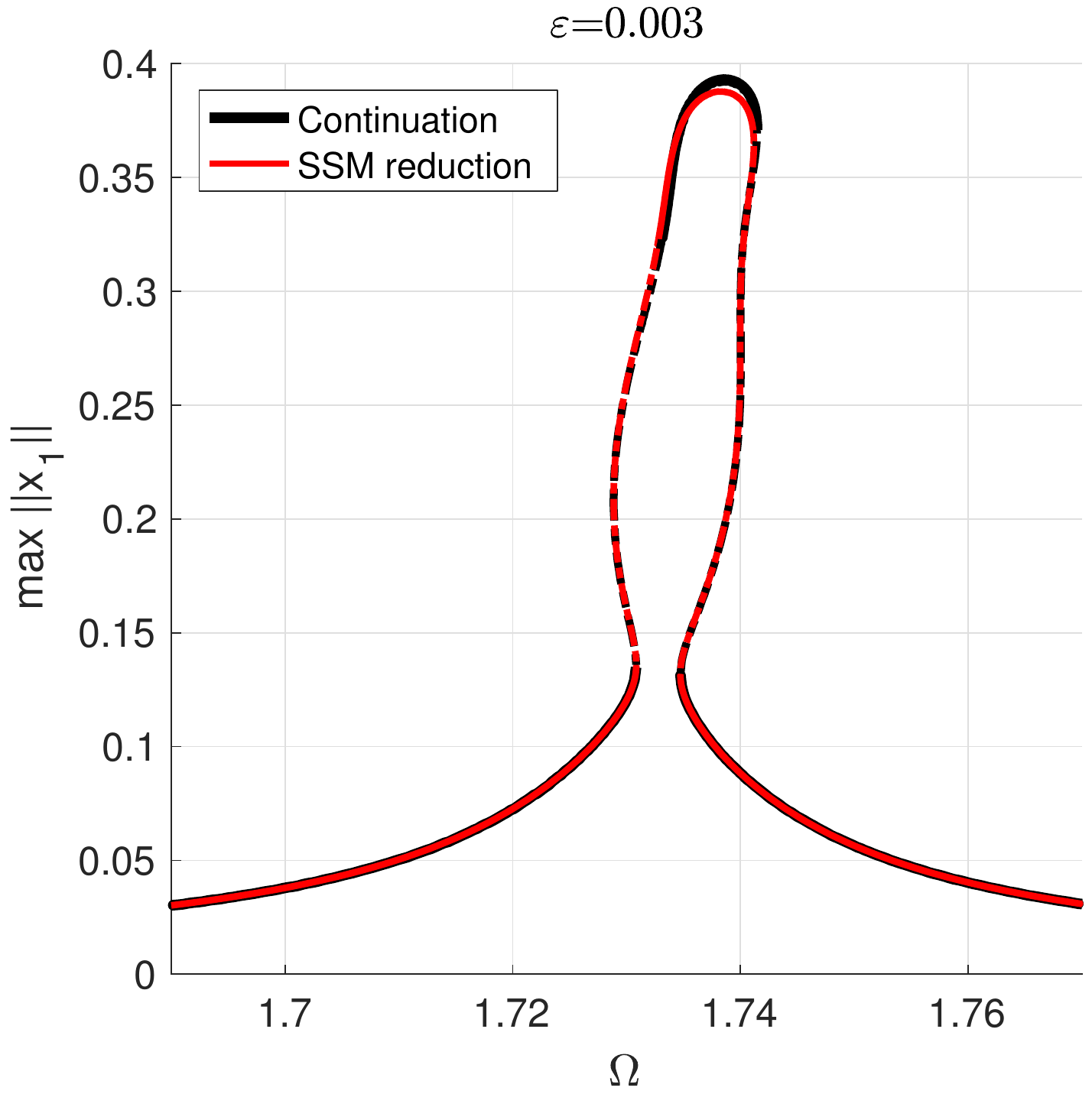}
        \fi
        \hspace*{8mm}
        \caption{\label{fig:FRC_SP_O5_3}}
        \vspace*{2mm}
    \end{subfigure}
    \caption{(a) Resulting forced response curve for $\varepsilon=0.001$ in Example \ref{ex:example_SP_2}. Two isolas are born out of the non-trivial transverse zeros of $a(\rho)$, located at ($b(\Omega),\rho$) on the autonomous backbone curve. The dashed lines indicate that the lower-amplitude isola is unstable in nature, whereas the higher-amplitude isola is partially stable. (b) Forced response curve for $\varepsilon=0.0025$. Both isolas have merged into one bigger isolated region. The lower half of the merged isolas is unstable in nature, whereas the upper half is stable. (c) Forced response curve for $\varepsilon=0.003$. The two merged isolas now have merged with the lower branch of the FRC. \label{fig:FRC_SP_O5}}
\end{figure}
\end{example}

\subsection{A discretized, forced Bernoulli beam with a cubic spring \label{sec:beam}}
Inspired by a similar example of Peeters et al. \cite{Peeters2009}, we now construct a reduced-order model for a discretized, cantilevered Bernoulli beam with a cubic spring and damper attached to the free end of the beam. We obtain the reduced model by computing the dynamics on the slowest, two-dimensional, time-periodic SSM of the system. This example is complex enough so that numerical continuation for obtaining the forced response is no longer feasible numerically, whereas our SSM-based predictions can still be carried out and will be seen to be accurate. 

We consider a square beam of length $L$, with cross-section $A$, situated in a Cartesian coordinate system of $(x,y,z)$ and basis $(\mtrx{e}_x,\mtrx{e}_y,\mtrx{e}_z)$. We list the relevant beam parameters in Table \ref{tab:system_par_beam}. 
\begin{table}
\begin{centering}
\caption{Notation used in the discretized beam example. \label{tab:system_par_beam}}
\begin{tabular}{|c|c|}
\hline 
Symbol & Meaning {(}unit{)}\tabularnewline
\hline 
\hline 
$L$ & Length of beam $(\unit[]{mm})$\tabularnewline
\hline 
$h$ & Height of beam $(\unit[]{mm})$\tabularnewline
\hline 
$b$ & Width of beam $(\unit[]{mm})$\tabularnewline
\hline 
$\rho$ & Density $(\unit[]{kg/mm^3})$\tabularnewline
\hline 
$E$ & Young's Modulus $(\unit[]{kPa})$\tabularnewline
\hline 
$I$ & Area moment of inertia $(\unit[]{mm^4})$\tabularnewline
\hline 
$\kappa$ & Coefficient cubic spring $(\unit[]{N/mm^3})$ \tabularnewline
\hline 
$\gamma$ & Coefficient cubic damper $(\unit[]{N \cdot s/mm^3})$ \tabularnewline
\hline 
$A$ & Cross-section of beam $(\unit[]{mm^2})$\tabularnewline
\hline 
$P$ & External forcing amplitude $(\unit[]{N})$\tabularnewline
\hline
\end{tabular}
\par\end{centering}
\end{table}
The line of points coinciding with the $x$-axis is called the beam's neutral axis. The Bernoulli hypothesis states that initially straight material lines, normal to the neutral axis, remain (a) straight and (b) inextensible, and (c) rotate as rigid lines to remain perpendicular to the beam's neutral axis after deformation. These kinematic assumptions are satisfied the displacement field,
\begin{align}
u_x(x,y,z,t) & = -z\frac{\partial w(x,t)}{\partial x}, \\
u_y(x,y,z,t) & = 0, \\
u_z(x,y,z,t) & = w(x,t),
\end{align}
where $(u_x,u_y,u_z)$ are the components of the displacement field $\mtrx{u}(x,y,z,t)$ of a material point located at $(x,y,z)$. The transverse displacement of a material point with initial coordinates on the beam's neutral axis at $z=0$ is denoted by $w(x)$. The rotation angle of a transverse normal line about the $y$-axis is given by $-\partial_x w(x)$. 
Using the Green-Lagrange strain tensor, we can express the relevant strains as
\begin{align}
\varepsilon_{xx} = -z \frac{\partial^2 w(x,t)}{\partial x^2}, \quad
\gamma_{xz} = 2\varepsilon_{xz} = 0.
\end{align}
We assume an isotropic, linearly elastic constitutive relation between the stresses and strains, i.e.
\begin{equation}
\sigma_{xx} = E\varepsilon_{xx},
\end{equation}
which finally leads to the equation of motion of the beam
\begin{equation}
\rho A\frac{\partial^2 w(x,t)}{\partial{t^2}}-\rho I \frac{\partial^4 w(x,t)}{\partial x^2 \partial t^2} + EI \frac{\partial^4 w(x,t)}{\partial x^4} = 0. \label{eq:PDE_beam}
\end{equation}
We assume that the thickness of the beam is small compared to its length, i.e., $h \ll L$, and hence we can neglect the mixed partial derivative term in Eq. (\ref{eq:PDE_beam}) (cf. Reddy and Mahaffey \cite{Reddy2013}).

After discretization of (\ref{eq:PDE_beam}), we obtain the set of ordinary differential equations
\begin{equation}
\mtrx{M}\ddot{\mtrx{x}}+\mtrx{K}\mtrx{x}=\mtrx{0},
\end{equation} 
where $\mtrx{x}\in\mathbb{R}^{2m}$, and $m$ is the number of elements used in the discretization. Each node of the beam has two coordinates related to the transverse displacement $w(x)$ and the rotation angle $-\partial_x w(x)$ of the cross section. We assume structural damping by considering the damping matrix
\begin{equation}
\mtrx{C} = \alpha \mtrx{M} + \beta \mtrx{K},
\end{equation}
with parameters $\alpha$ and $\beta$. We apply cosinusoidal external forcing on the transverse displacement coordinate at the free end of the beam with forcing frequency $\Omega$ and forcing amplitude $\varepsilon P$. Additionally, we add a cubic spring and damper along this coordinate, with coefficients $\kappa$ and $\gamma$, respectively. As a result, the equations of motion of the beam can be written as
\begin{equation}
\mtrx{M}\ddot{\mtrx{x}}+\mtrx{C}\dot{\mtrx{x}}+\mtrx{K}\mtrx{x}+\mtrx{g}(\mtrx{x},\dot{\mtrx{x}})=\varepsilon\mtrx{f}(\Omega t).
\end{equation} 
We illustrate the kinematics, the forcing and the cubic spring and damper in Fig. \ref{fig:EB_beam}.
\begin{figure}[h!]
\centering
    \begin{subfigure}{1\textwidth}
        \centering
        \if\mycmd1
        \includegraphics[scale=1]{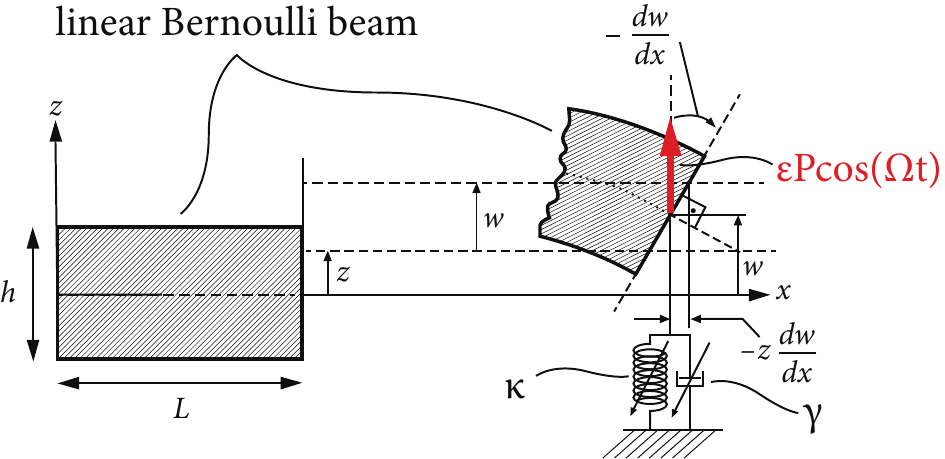}
        \fi
    \end{subfigure}
    \caption{Forced Bernoulli beam with a cubic spring and damper. \label{fig:EB_beam}}
\end{figure}

We select $m=25$ for the number of discretized elements, which gives $\mtrx{x}\in\mathbb{R}^{50}$, resulting in a 100-dimensional phase space. We list the geometric and material parameter values in Table \ref{tab:system_par_beam_ex_1}.

\begin{table}[H]
\begin{centering}
\caption{Geometric and material parameters for the Bernoulli beam. \label{tab:system_par_beam_ex_1}}
\begin{tabular}{|c|c|}
\hline 
Symbol & Value \tabularnewline
\hline 
\hline 
$L$ & $\unit[2700]{mm}$\tabularnewline
\hline 
$h$ & $\unit[10]{mm}$ \tabularnewline
\hline 
$b$ & $\unit[10]{mm}$ \tabularnewline
\hline 
$\rho$ & $\unit[1780\cdot 10^{-9}]{kg/mm^3}$ \tabularnewline
\hline 
$E$ & $\unit[45\cdot 10^6]{kPa}$ \tabularnewline
\hline 
$\kappa$ & $\unit[6]{N/mm^3}$ \tabularnewline
\hline 
$\gamma$ & $\unit[-0.02]{N \cdot s/mm^3}$  \tabularnewline
\hline 
$\alpha$ & $\unit[1.25\cdot 10^{-4}]{}$ \tabularnewline
\hline 
$\beta$ & $\unit[2.5 \cdot 10^{-4}]{}$ \tabularnewline
\hline
$P$ & $\unit[0.1]{N}$ \tabularnewline
\hline
\end{tabular}
\par\end{centering}
\end{table}
For these parameter values, the eigenvalues corresponding to the slowest eigenspace are 
\begin{equation}
\lambda_{1,2}=-0.0061884 \pm 7.0005 \mi.
\end{equation}
As earlier, introducing the scaling $\mtrx{s}\rightarrow\varepsilon^{\frac{1}{4}}\mtrx{s}$, we obtain the approximations
\begin{align}
a(\rho) &= -0.0061884\rho + 0.036202\rho^3,  \label{eq:beam_auto_a}\\
b(\rho) &=  7.0005 + 0.031689\rho^2, \\
c_{1,\mtrx{0}} &= 0.54645 + 0.00048\mi.
\end{align}
The function $a(\rho)$ in Eq. (\ref{eq:beam_auto_a}) has a non-trivial, transverse, positive zero at $\rho_1^+ = 0.413$. Fig. \ref{fig:FRC_beam_roots} shows this zero to be non-spurious. Therefore, by Theorem \ref{thrm:isola}, an isola will perturb from the point $(\Omega=b(\rho_1^+),\rho_1^+)$ of the autonomous backbone curve. Also by Theorem \ref{thrm:isola}, the isola will merge with the main branch of the FRC approximately for
\begin{equation}
\varepsilon_\text{m} =  \frac{1}{\norm{c_{1,\mtrx{0}}}}\sqrt{\frac{4|\text{Re}(\lambda_1)|^3}{27 \text{Re}(\gamma_1)}}=0.0018. \label{eq:beam_merge}
\end{equation}
\begin{figure}
\centering
    \begin{subfigure}{0.5\textwidth}
        \centering
     	\if\mycmd1        
        \includegraphics[scale=0.45]{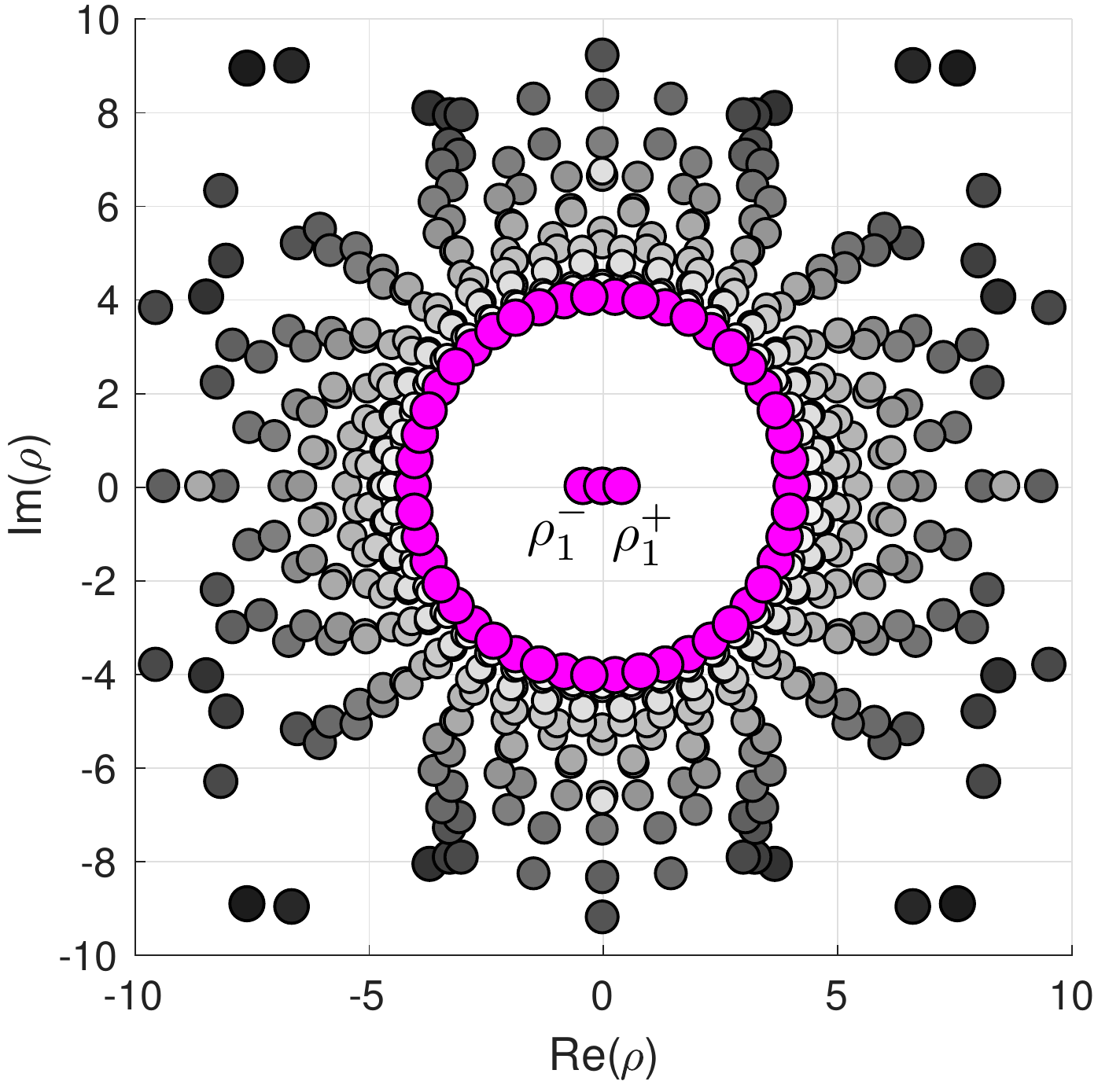}
        \fi
    \end{subfigure}
    \caption{Plot of the roots of $a(\rho)$ in the complex plane for Example \ref{sec:beam}, with brighter colors indicating an increasing order $M$ in the expansion of $a(\rho)$, up to 50th order (roots that are negative of each other are to be identified). The zeros from the highest approximation are highlighted in magenta. We observe that a non-trivial transverse zero $\rho_1^\pm$ persists for higher-order approximations and is clearly within the domain of convergence of the function $a(\rho)$.\label{fig:FRC_beam_roots}}
\end{figure}
To verify our predicted merger amplitude in (\ref{eq:beam_merge}), we perform a discrete numerical sweep of the full system, in which we force the system at different forcing frequencies and plot the resulting maximum absolute value of the transverse displacement of the tip of the beam, while keeping the forcing amplitude fixed (see Fig. \ref{fig:FRC_beam_ex_sweep}). 

While numerical continuation remains a powerful tool in verifying our analytic predictions in lower dimensions, its use becomes unfeasible in higher dimensions. For this reason, Fig. \ref{fig:FRC_beam_ex_sweep} only shows a discrete set of periodic responses computed from a point-wise, long-term numerical integration leading to a steady state, as opposed to a continuous FRC obtained from numerical continuation. 

\begin{figure}

\centering
        \begin{subfigure}{0.45\textwidth}
        \centering
        \if\mycmd1
        \includegraphics[scale=0.36]{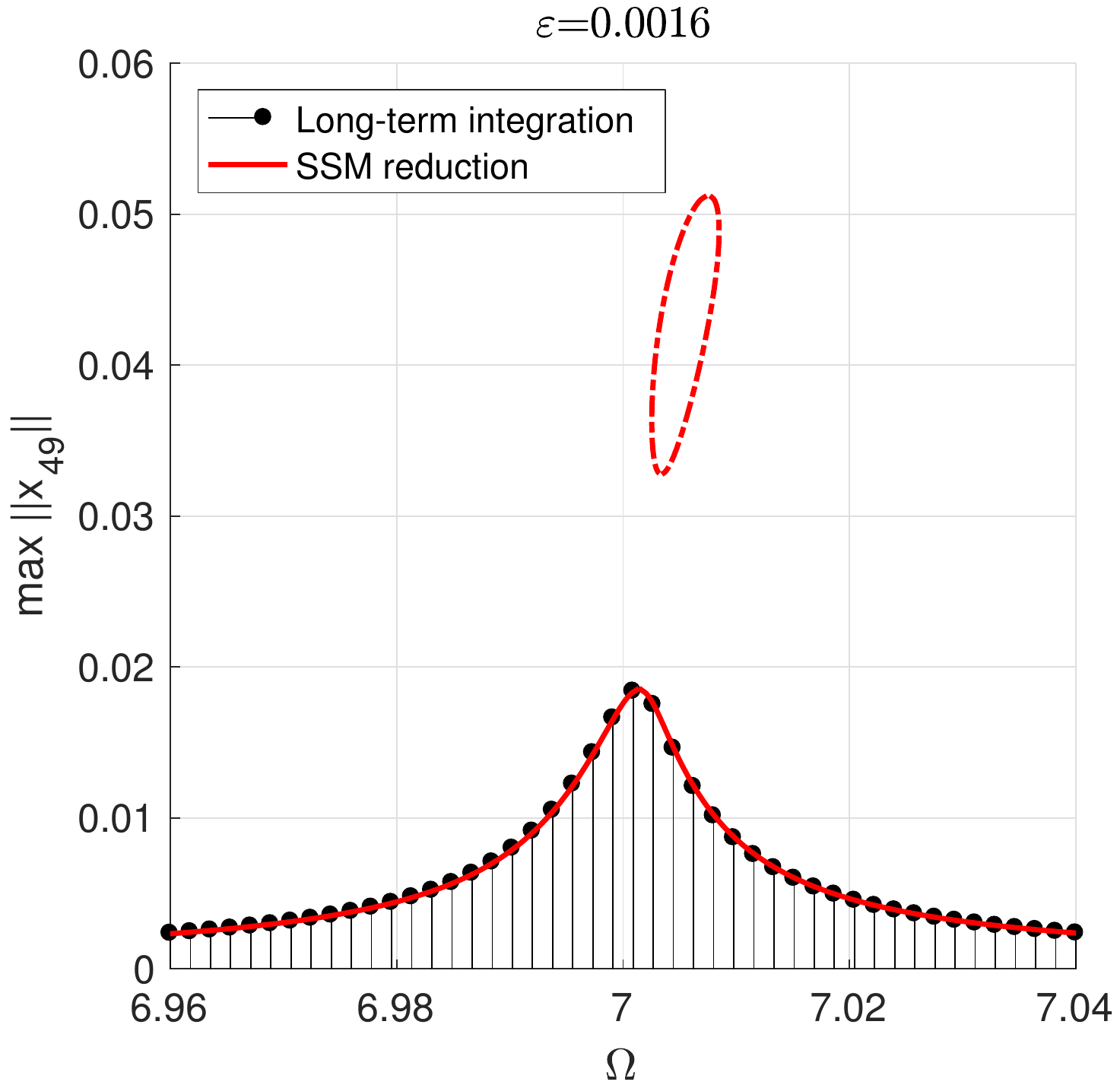}
        \fi
        \hspace{6mm}
        \caption{}
    \end{subfigure}
    \begin{subfigure}{0.45\textwidth}
        \centering
        \if\mycmd1
        \includegraphics[scale=0.36]{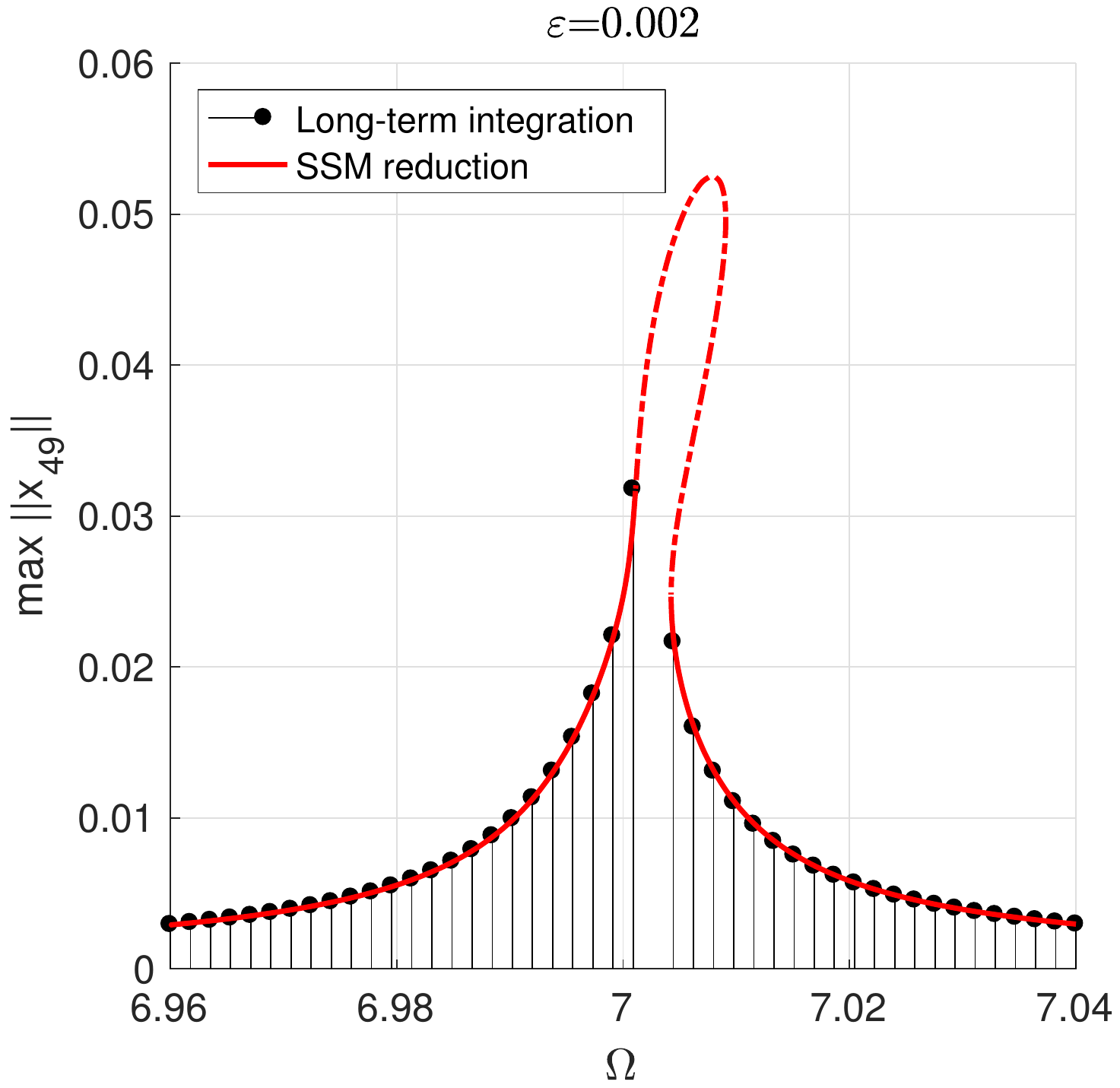}
        \fi
        \hspace{6mm}
        \caption{}
    \end{subfigure}
    \caption{(a) Extracted forced response curve for $\varepsilon=0.0016$ in Example \ref{sec:beam}. An unstable isola is born out of the non-trivial transverse zero of $a(\rho)$, located at $(\Omega=b(\rho_1^+),\rho_1^+)$ on the autonomous backbone curve. (b) Extracted forced response curve from the reduced dynamics for $\varepsilon=0.002>\varepsilon_\text{m}$. The main FRC branch has merged with the unstable isola. A discrete frequency sweep has been performed on the full 100-dimensional system to verify the accuracy of our two-dimensional reduced model. The frequency region in which the FRC becomes unstable, as predicted by the SSM-based reduced dynamics, is confirmed by the full numerical frequency sweep.  \label{fig:FRC_beam_ex_sweep}}
\end{figure}

\section{Conclusion}
We have used the exact reduced dynamics on two-dimensional time-periodic spectral submanifolds (SSMs) to extract forced-response curves (FRCs) and predict isolas in arbitrary multi-degree-of-freedom mechanical systems without performing costly numerical simulations.  We showed that for a cubic-order approximation, the reduced dynamics on the SSM gives an analytic prediction for the isolas, valid for any mode of a multi-degree-of-freedom oscillatory system. For simple examples, these predictions can explicitly be expressed as functions of the system parameters. Our lower-order predictions can be refined to higher-orders using the publicly available \textsc{MATLAB} script \textsc{SSMtool}\footnote{\textsc{SSMtool} is available at: \href{http://www.georgehaller.com}{www.georgehaller.com}}.

For mechanical systems of high degrees of freedom, numerical continuation techniques for forced response curves become computationally expensive. Instead, using the non-autonomous SSM and the corresponding reduced dynamics on the SSM, we are able to approximate all possible FRCs for different forcing amplitudes, as our expressions depend symbolically on the forcing amplitude $\varepsilon$. An additional advantage of the results derived here is that the isolas are uncovered by the transverse intersection of the zero-level sets of our two reduced equations. The isolas will generally be missed by numerical continuation techniques, which require starting on an isolated solution branch. As we have shown, our predictions for the main FRC branches, as well as for isolas, remain valid and computable in high dimensional problems in which numerical continuation is no longer a viable alternative for constructing these curves. Using the general results of Haller and Ponsioen \cite{Haller2016}, one can extend the periodic approach to detect quasi-periodic responses and isolas under quasi-periodic forcing.

\appendix
\section{Coefficient equations for the the non-autonomous SSMs \label{app:coeff_eq}}
As the SSM perturbs smoothly with respect to the parameter $\varepsilon$, we can expand $\mtrx{W}(\mtrx{s},\phi)$ and $\mtrx{R}(\mtrx{s},\phi)$ in $\varepsilon$, obtaining the expressions
\begin{align}
\mtrx{W}(\mtrx{s},\phi)&=\mtrx{W}_0(\mtrx{s})+\varepsilon\mtrx{W}_1(\mtrx{s},\phi)+\mathcal{O}(\varepsilon^2), \label{eq:SSM_exp}\\
\mtrx{R}(\mtrx{s},\phi)&=\mtrx{R}_0(\mtrx{s})+\varepsilon\mtrx{R}_1(\mtrx{s},\phi)+\mathcal{O}(\varepsilon^2).\label{eq:RED_exp}
\end{align}
The forcing terms in system (\ref{eq:ds_diag}) are of order $\varepsilon$, therefore, the leading-order terms in Eqs. (\ref{eq:SSM_exp})-(\ref{eq:RED_exp}) are not functions of $\phi$. We now substitute Eqs. (\ref{eq:SSM_exp})-(\ref{eq:RED_exp}) into the invariance Eq. (\ref{eq:invar}) and collect terms of equal order in $\varepsilon$. Given that $\mtrx{G}_\text{m}(\mtrx{q})=\mathcal{O}(\left|\mtrx{q}\right|^2)$, we can Taylor expand $\mtrx{G}_\text{m}(\mtrx{W}(\mtrx{s},\phi))$ around $\varepsilon=0$, to obtain
\begin{equation}
\mtrx{G}_\text{m}(\mtrx{W}(\mtrx{s},\phi)) = \mtrx{G}_\text{m}(\mtrx{W}_0(\mtrx{s}))+\varepsilon D_\mtrx{q}\mtrx{G}_\text{m}(\mtrx{W}_0(\mtrx{s}))\mtrx{W}_1(\mtrx{s},\phi)+\mathcal{O}(\varepsilon^2).
\end{equation}
Collecting terms of $\mathcal{O}(1)$, we obtain
\begin{equation}
\gmtrx{\Lambda} \mtrx{W}_0(\mtrx{s}) + \mtrx{G}_\text{m}(\mtrx{W}_0(\mtrx{s})) = D_\mtrx{s} \mtrx{W}_0(\mtrx{s})\mtrx{R}_0(\mtrx{s}),\label{eq:auto_SSM}
\end{equation}
and, subsequently, collecting terms of $\mathcal{O}(\varepsilon)$ we find that
\begin{align}
&\gmtrx{\Lambda} \mtrx{W}_1(\mtrx{s},\phi) + D_\mtrx{q}\mtrx{G}_\text{m}(\mtrx{W}_0(\mtrx{s}))\mtrx{W}_1(\mtrx{s},\phi) + \mtrx{F}_\text{m}(\phi) \label{eq:invar_O_eps}\\ & = D_\mtrx{s} \mtrx{W}_0(\mtrx{s})\mtrx{R}_1(\mtrx{s},\phi) 
+ D_\mtrx{s} \mtrx{W}_1(\mtrx{s},\phi)\mtrx{R}_0(\mtrx{s}) + D_{\phi} \mtrx{W}_1(\mtrx{s},\phi)\Omega. \nonumber
\end{align}
Eq. (\ref{eq:auto_SSM}) is derived and solved in Ponsioen et al. \cite{ponsioen2018automated} to compute autonomous SSMs using the Kronecker product. The construction of the autonomous SSMs has been automated and cast into the \textsc{MATLAB}-based computational tool \textsc{SSMtool}. In the following sections, the $\mathcal{O}(\varepsilon)$ terms of the SSM will simply be referred to as the non-autonomous part of the SSM. Note that in order to solve Eq. (\ref{eq:invar_O_eps}), Eq. (\ref{eq:auto_SSM}) needs to be solved first for the autonomous SSM and for the reduced dynamics, since the respective terms obtained from Eq. (\ref{eq:auto_SSM}) are needed in Eq. (\ref{eq:invar_O_eps}).

The autonomous and non-autonomous parts of the SSM and the reduced dynamics, which have previously been derived from an expansion in $\varepsilon$, are in turn Taylor expanded in the parameterization coordinates $\mtrx{s}$, which we explicitly express as 
\begin{align}
\mtrx{W}_0(\mtrx{s}) &=
\begin{bmatrix}
w_1^0(\mtrx{s}) \\
\vdots \\
w_{2n}^0(\mtrx{s}) 
\end{bmatrix},\quad w_i^0(\mtrx{s})=\sum_{\mtrx{m}}W_{i,\mtrx{m}}^0\mtrx{s}^\mtrx{m},  \\
\mtrx{R}_0(\mtrx{s}) &=
\begin{bmatrix}
r_1^0(\mtrx{s}) \\
\vdots \\
r_{2n}^0(\mtrx{s}) 
\end{bmatrix},\quad r_i^0(\mtrx{s})=\sum_{\mtrx{m}}R_{i,\mtrx{m}}^0\mtrx{s}^\mtrx{m}, \\
\mtrx{W}_1(\mtrx{s},\phi) &=
\begin{bmatrix}
w_1^1(\mtrx{s},\phi) \\
\vdots \\
w_{2n}^1(\mtrx{s},\phi) 
\end{bmatrix},\quad w_i^1(\mtrx{s},\phi)=\sum_{\mtrx{m}}W_{i,\mtrx{m}}^1(\phi)\mtrx{s}^\mtrx{m}, \\
\mtrx{R}_1(\mtrx{s},\phi) &=
\begin{bmatrix}
r_1^1(\mtrx{s},\phi) \\
\vdots \\
r_{2n}^1(\mtrx{s},\phi) 
\end{bmatrix},\quad r_i^1(\mtrx{s},\phi)=\sum_{\mtrx{m}}R_{i,\mtrx{m}}^1(\phi)\mtrx{s}^\mtrx{m}. 
\end{align}
Here we have made us of the multi-index notation $\mtrx{m}\in\mathbb{N}_0^{2}$. We now assume that Eq. (\ref{eq:auto_SSM}) has already been solved for $\mtrx{W}_0(\mtrx{s})$ and $\mtrx{R}_0(\mtrx{s})$ and therefore their coefficients $W_{i,\mtrx{m}}^0$ and $R_{i,\mtrx{m}}^0$ are known.
\begin{theorem}\label{thrm:coef_eq}
For $\phi\in S^1$, the coefficient equation related to the $\textbf{k}^\text{th}$-power term of the $i^\text{th}$ row of the non-autonomous invariance Eq. (\ref{eq:invar_O_eps}), is equal to
\begin{equation}
\left(\lambda_i-\sum_{j=1}^{2}k_j\lambda_j\right)W_{i,\mtrx{k}}^1(\phi)-D_{\phi} W_{i,\mtrx{k}}^1(\phi)\Omega = \sum_{j=1}^{2}\delta_{ij}R_{j,\mtrx{k}}^1(\phi) + P_{i,\mtrx{k}}(\phi), \label{eq:coef_EQ_W1}
\end{equation}
where $P_{i,\mtrx{k}}(\phi)$ can be written as 
\begin{align}
& \sum_{j=1}^{2}\sum_{\substack{\mtrx{m}\leq\tilde{\mtrx{k}}_j \\ \mtrx{m}\neq \mtrx{e}_j \\ m_j>0}} m_j W_{i,\mtrx{m}}^0 R_{j,\tilde{\mtrx{k}}_j-\mtrx{m}}^1(\phi)+\sum_{j=1}^{2}\sum_{\substack{\mtrx{m}\leq\tilde{\mtrx{k}}_j \\ \mtrx{m}\neq\mtrx{k} \\ m_j>0}} m_j W_{i,\mtrx{m}}^1(\phi)R_{j,\tilde{\mtrx{k}}_j-\mtrx{m}}^0 \\
&-F_{i,\mtrx{k}}(\phi) -\left[\sum_{j=1}^{2n}D_{x_j}g_i(\mtrx{W}_0(\mtrx{s}))w_j^1(\mtrx{s},\phi)\right]_\mtrx{k} \nonumber
\end{align}
\end{theorem}
\begin{proof}
Assuming that $\phi\in S^1$, we obtain that for the $i^\text{th}$ row, the $\mtrx{k}^\text{th}$-power terms on the right-hand side of Eq. (\ref{eq:invar_O_eps}) can be expressed as
\begin{align}
\left[D_\mtrx{s}\mtrx{W}_0(\mtrx{s})\mtrx{R}_1(\mtrx{s},\phi)\right]^{\mtrx{k}}_i &= \sum_{j=1}^{2}\sum_{\substack{\mtrx{m}\leq\tilde{\mtrx{k}}_j \\ m_j>0}} m_j W_{i,\mtrx{m}}^0 R_{j,\tilde{\mtrx{k}}_j-\mtrx{m}}^1(\phi), \\
\left[D_\mtrx{s}\mtrx{W}_1(\mtrx{s},\phi)\mtrx{R}_0(\mtrx{s})\right]^{\mtrx{k}}_i &= \sum_{j=1}^{2}\sum_{\substack{\mtrx{m}\leq\tilde{\mtrx{k}}_j \\ m_j>0}} m_j W_{i,\mtrx{m}}^1(\phi)R_{j,\tilde{\mtrx{k}}_j-\mtrx{m}}^0, \\
\left[D_{\phi} \mtrx{W}_1(\mtrx{s},\phi)\Omega\right]^{\mtrx{k}}_i &=D_{\phi}W_{i,\mtrx{k}}^1(\phi)\Omega ,
\end{align}
where we have made use of the multi-index notation
\begin{equation}
\mtrx{m}\in\mathbb{N}_0^2,\quad \mtrx{k}\in\mathbb{N}_0^2,\quad \tilde{\mtrx{k}}_j = \mtrx{k}+\mtrx{e}_j,
\end{equation}
with $\mtrx{e}_j$ denoting a unit vector. 

The $\mtrx{k}^\text{th}$-power terms on the left-hand side of the $i^\text{th}$ row of Eq. (\ref{eq:invar_O_eps}) can be written as
\begin{align}
\left[\gmtrx{\Lambda}\mtrx{W}_1(\mtrx{s},\phi)\right]^\mtrx{k}_i &= \lambda_iW_{i,\mtrx{k}}^1(\phi), \\
\left[D_\mtrx{x}\mtrx{G}_\text{m}(\mtrx{W}_0(\mtrx{s}))\mtrx{W}_1(\mtrx{s},\phi)\right]^{\mtrx{k}}_i &=\left[\sum_{j=1}^{2n}D_{x_j}g_i(\mtrx{W}_0(\mtrx{s}))w_j^1(\mtrx{s},\phi)\right]_\mtrx{k}, \\
\left[\mtrx{F}_\text{m}(\phi)\right]^\mtrx{k}_i&=F_{i,\mtrx{k}}(\phi). 
\end{align}
Therefore, the coefficient equation related to the $\textbf{k}^\text{th}$-power term of the $i^\text{th}$ row of the non-autonomous invariance Eq. (\ref{eq:invar_O_eps}) is
\begin{equation}
\left(\lambda_i-\sum_{j=1}^{2}k_j\lambda_j\right)W_{i,\mtrx{k}}^1(\phi)-D_{\phi} W_{i,\mtrx{k}}^1(\phi)\Omega = \sum_{j=1}^{2}\delta_{ij}R_{j,\mtrx{k}}^1(\phi) + P_{i,\mtrx{k}}(\phi),
\end{equation}
where 
\begin{align}
P_{i,\mtrx{k}}(\phi) &= \sum_{j=1}^{2}\sum_{\substack{\mtrx{m}\leq\tilde{\mtrx{k}}_j \\ \mtrx{m}\neq \mtrx{e}_j \\ m_j>0}} m_j W_{i,\mtrx{m}}^0 R_{j,\tilde{\mtrx{k}}_j-\mtrx{m}}^1(\phi)+\sum_{j=1}^{2}\sum_{\substack{\mtrx{m}\leq\tilde{\mtrx{k}}_j \\ \mtrx{m}\neq\mtrx{k} \\ m_j>0}} m_j W_{i,\mtrx{m}}^1(\phi)R_{j,\tilde{\mtrx{k}}_j-\mtrx{m}}^0 \\
&-F_{i,\mtrx{k}}(\phi) -\left[\sum_{j=1}^{2n}D_{x_j}g_i(\mtrx{W}_0(\mtrx{s}))w_j^1(\mtrx{s},\phi)\right]_\mtrx{k}, \nonumber
\end{align}
which concludes the proof of Theorem \ref{thrm:coef_eq}.\qed
\end{proof}
\subsection{Solving the non-autonomous invariance equation for $|\mtrx{k}|=0$ \label{sec:sol_k_0}}
For $|\mtrx{k}|=\mtrx{0}$, Eq. (\ref{eq:coef_EQ_W1}) reduces to 
\begin{equation}
\lambda_iW_{i,\mtrx{0}}^1(\phi)-D_{\phi} W_{i,\mtrx{0}}^1(\phi)\Omega = \sum_{j=1}^{2}\delta_{ij}R_{j,\mtrx{0}}^1(\phi) -F_{i,\mtrx{0}}(\phi), \label{eq:coef_EQ_W1_k_0}
\end{equation}
which is equivalent to the $0^\text{th}$-order expansion of Breunung and Haller \cite{Breunung2017}. Assuming that the forcing term $F_{i,\mtrx{0}}(\phi)$ can be written as 
\begin{equation}
F_{i,\mtrx{0}}(\phi)= \tilde{F}_{i,\mtrx{0}}\frac{\me^{\mi\phi}+\me^{-\mi\phi}}{2},
\end{equation}
we express $W_{i,\mtrx{0}}^1(\phi)$ and $R_{i,\mtrx{0}}^1(\phi)$ in the following form
\begin{equation}
W_{i,\mtrx{0}}^1(\phi) = a_{i,\mtrx{0}}\me^{\mi\phi} + b_{i,\mtrx{0}}\me^{-\mi\phi},\quad
R_{i,\mtrx{0}}^1(\phi) = c_{i,\mtrx{0}}\me^{\mi\phi} + d_{i,\mtrx{0}}\me^{-\mi\phi}.
\end{equation}
As the SSM is constructed over a two-dimensional spectral subpace $\mathcal{E}$, corresponding to the eigenvalues $\lambda_1=\text{Re}\lambda_1+\mi\text{Im}\lambda_1$ and $\lambda_2=\bar{\lambda}_1=\text{Re}\lambda_1-\mi\text{Im}\lambda_1$, we can write the solution of Eq. (\ref{eq:coef_EQ_W1_k_0}) as 
\begin{equation}
W_{i,\mtrx{0}}^1 = \frac{\delta_{i1}c_{1,\mtrx{0}}+\delta_{i2}c_{2,\mtrx{0}}-\frac{1}{2}\tilde{F}_{i,\mtrx{0}}}{\lambda_i-\mi\Omega}\me^{\mi\phi} + \frac{\delta_{i1}d_{1,\mtrx{0}}+\delta_{i2}d_{2,\mtrx{0}}-\frac{1}{2}\tilde{F}_{i,\mtrx{0}}}{\lambda_i+\mi\Omega}\me^{-\mi\phi}. \label{eq:coef_0}
\end{equation}
Our main goal is to obtain the forced response curve around the frequency $\text{Im}\lambda_1$, which corresponds to the spectral subspace $\mathcal{E}$. For lightly damped systems where $\text{Re}\lambda_1$ is small, we obtain small denominators in Eq. (\ref{eq:coef_0}) if the forcing frequency $\Omega$ is approximately equal to $\text{Im}\lambda_1$. We, therefore, intend to remove this near-resonance by setting
\begin{equation}
c_{1,\mtrx{0}}=\frac{1}{2}\tilde{F}_{1,\mtrx{0}},\quad c_{2,\mtrx{0}}=0,\quad
d_{1,\mtrx{0}}=0,\quad d_{2,\mtrx{0}}=\frac{1}{2}\tilde{F}_{2,\mtrx{0}}. 
\end{equation}
\subsection{Solving the non-autonomous invariance equation for $|\mtrx{k}|>0$}
For $|\mtrx{k}|>0$, the solution to the non-autonomous invariance Eq. (\ref{eq:coef_EQ_W1}) is given by 
\begin{equation}
W_{i,\mtrx{k}}^1(\phi)=\underbrace{\frac{\sum_{j=1}^{2}\delta_{ij}c_{j,\mtrx{k}}+\alpha_{i,\mtrx{k}}}{\lambda_i-\sum_{j=1}^{2}k_j\lambda_j-\mi\Omega}}_{a_{i,\mtrx{k}}}\me^{\mi\phi}+\underbrace{\frac{\sum_{j=1}^{2}\delta_{ij}d_{j,\mtrx{k}}+\beta_{i,\mtrx{k}}}{\lambda_i-\sum_{j=1}^{2}k_j\lambda_j+\mi\Omega}}_{b_{i,\mtrx{k}}}\me^{-\mi\phi}, \label{eq:coef_EQ_W1_k_higher}
\end{equation}
where we have let $P_{i,\mtrx{k}}=\alpha_{i,\mtrx{k}}\me^{\mi \phi}+\beta_{i,\mtrx{k}}\me^{-\mi \phi}$. 
Using the same reasoning as in section \ref{sec:sol_k_0}, we want to choose $c_{i,\mtrx{k}}$ and $d_{i,\mtrx{k}}$ in Eq. (\ref{eq:coef_EQ_W1_k_higher}) in such a way that the coefficients $a_{i,\mtrx{k}}$ and $b_{i,\mtrx{k}}$, which become large when the damping is small and the forcing frequency $\Omega$ is close to $\text{Im}\lambda_1$, are completely removed. These terms are the so-called resonant terms.  In Table \ref{tab:res_non_auto}, we list the resonant terms $a_{i,\mtrx{k}}$ and $b_{i,\mtrx{k}}$ for small damping and for $\Omega\approx\text{Im}\lambda_1$.
\begin{table}[H]
\caption{Resonant terms in the non-autonomous part of the SSM for small damping and forcing frequency $\Omega\approx\text{Im}\lambda_1$. \label{tab:res_non_auto}}
\begin{centering}
\begin{tabular}{|l|l|l|}
\hline
& $i=1$ & $i=2$\tabularnewline
\hline
\hline
$a_{i,\mtrx{k}}$ & $k_1=k_2$ & $k_1=k_2-2$ \tabularnewline
\hline
$b_{i,\mtrx{k}}$ & $k_1=k_2+2$ & $k_1=k_2$ \tabularnewline
\hline
\end{tabular}
\par\end{centering}
\end{table}
The terms listed in Table \ref{tab:res_non_auto} will be removed from the expressions of $\mtrx{W}_1(\mtrx{s},\phi)$ and included into the non-autonomous part of the reduced dynamics $\mtrx{R}_1(\mtrx{s},\phi)$, in order to avoid small denominators in $\mtrx{W}_1(\mtrx{s},\phi)$. 

\section{Proof of Theorem \ref{thm:red_dyn}\label{app:red_dyn}}
The $\mathcal{O}(\varepsilon)$ approximation of the reduced dynamics for $\mtrx{s}$ can be written as
\begin{equation}
\dot{\mtrx{s}} = \mtrx{R}(\mtrx{s},\phi)=\mtrx{R}_0(\mtrx{s})+\varepsilon\mtrx{R}_1(\mtrx{s},\phi),\label{eq:red_eq_total}
\end{equation}
where the first row of Eq. (\ref{eq:red_eq_total}) takes the form
\begin{align}
\dot{s}_1 &= \lambda_1 s_1 + \sum_{i=1}^M\gamma_is_1^{i+1}s_2^{i} \\
&+\varepsilon\left(c_{1,\mtrx{0}}\me^{\mi\phi} + \sum_{i=1}^M\left(c_{1,(i,i)}(\Omega)s_1^is_2^i\me^{\mi\phi}+d_{1,(i+1,i-1)}(\Omega)s_1^{i+1}s_2^{i-1}\me^{-\mi\phi}\right) \right), \nonumber
\end{align}
and the second row of Eq. (\ref{eq:red_eq_total}) is, by construction, the complex conjugate of the first row.
Introducing a change to polar coordinates, $s_1=\rho\me^{\mi\theta}$, $s_2=\bar{s}_1=\rho\me^{-\mi\theta}$, dividing by $\me^{\mi\theta}$ and introducing the new phase coordinate $\psi = \theta-\phi$, we obtain
\begin{align}
\dot{\rho}+\mi\rho(\dot{\psi}+\Omega) &= \lambda_1 \rho + \sum_{i=1}^M\gamma_i\rho^{2i+1} \label{eq:polar_red_dyn} \\
&+\varepsilon\left(c_{1,\mtrx{0}}\me^{-\mi\psi} + \sum_{i=1}^M\left(c_{1,(i,i)}(\Omega)\rho^{2i}\me^{-\mi\psi}+d_{1,(i+1,i-1)}(\Omega)\rho^{2i}\me^{\mi\psi}\right) \right). \nonumber
\end{align}
We obtain the result listed in Theorem \ref{thm:red_dyn} by splitting Eq. (\ref{eq:polar_red_dyn}) into its real and imaginary part. \qed

\section{Extracting the forced response curve \label{app:ext_fr}}
For convenience, we restate our zero problem (\ref{eq:zeroproblem}),
\begin{equation}
\mtrx{F}(\mtrx{u}) = 
\begin{bmatrix}
F_1(\mtrx{u})\\
F_2(\mtrx{u})
\end{bmatrix} = 
\begin{bmatrix}
a(\rho) + \varepsilon\left(f_1(\rho,\Omega)\cos(\psi)+f_2(\rho,\Omega)\sin(\psi)\right) \\
(b(\rho)-\Omega)\rho + \varepsilon\left(g_1(\rho,\Omega)\cos(\psi)-g_2(\rho,\Omega)\sin(\psi)\right) 
\end{bmatrix} = \mtrx{0}, \label{eq:zeroproblem_app}\\
\end{equation}
where
\begin{equation}
\mtrx{F}(\mtrx{u}):\mathbb{R}^3 \rightarrow \mathbb{R}^2, \quad
\mtrx{u} = 
\begin{bmatrix}
\rho \\
\Omega\\ 
\psi
\end{bmatrix}. \nonumber
\end{equation}
If there exists a regular point  $\mtrx p=(\rho,\Omega,\psi)$, such that $\mtrx F(\mtrx p )=\mtrx 0$  in (\ref{eq:zeroproblem_app}) and the Jacobian of $\mtrx F$ evaluated at $\mtrx p$ is surjective, then by the implicit function theorem, locally there exists a one-dimensional submanifold of $\mathbb{R}^3$ which will represent the forced response curve when projected onto the $(\Omega,\rho)$-space. Equivalently, the zero-level sets of $F_1(\mtrx u)$ and $F_2(\mtrx u)$ in (\ref{eq:zeroproblem_app}), which we will denote by $\mathcal{M}_1^{\mtrx p}$ and $\mathcal{M}_2^{\mtrx p}$, will be two two-dimensional submanifolds in the $(\rho,\psi,\Omega)$-space that, locally around $\mtrx p$, intersect each other transversely, yielding the forced response curve. We illustrate this concept in Fig. \ref{fig:FRC_intersection}, which is a typical picture for a damped non-linear periodically forced mechanical system with a hardening nonlinearity.

\begin{figure}
\centering
    \begin{subfigure}{0.45\textwidth}
        \centering
         \if\mycmd1
           \includegraphics[scale=0.35]{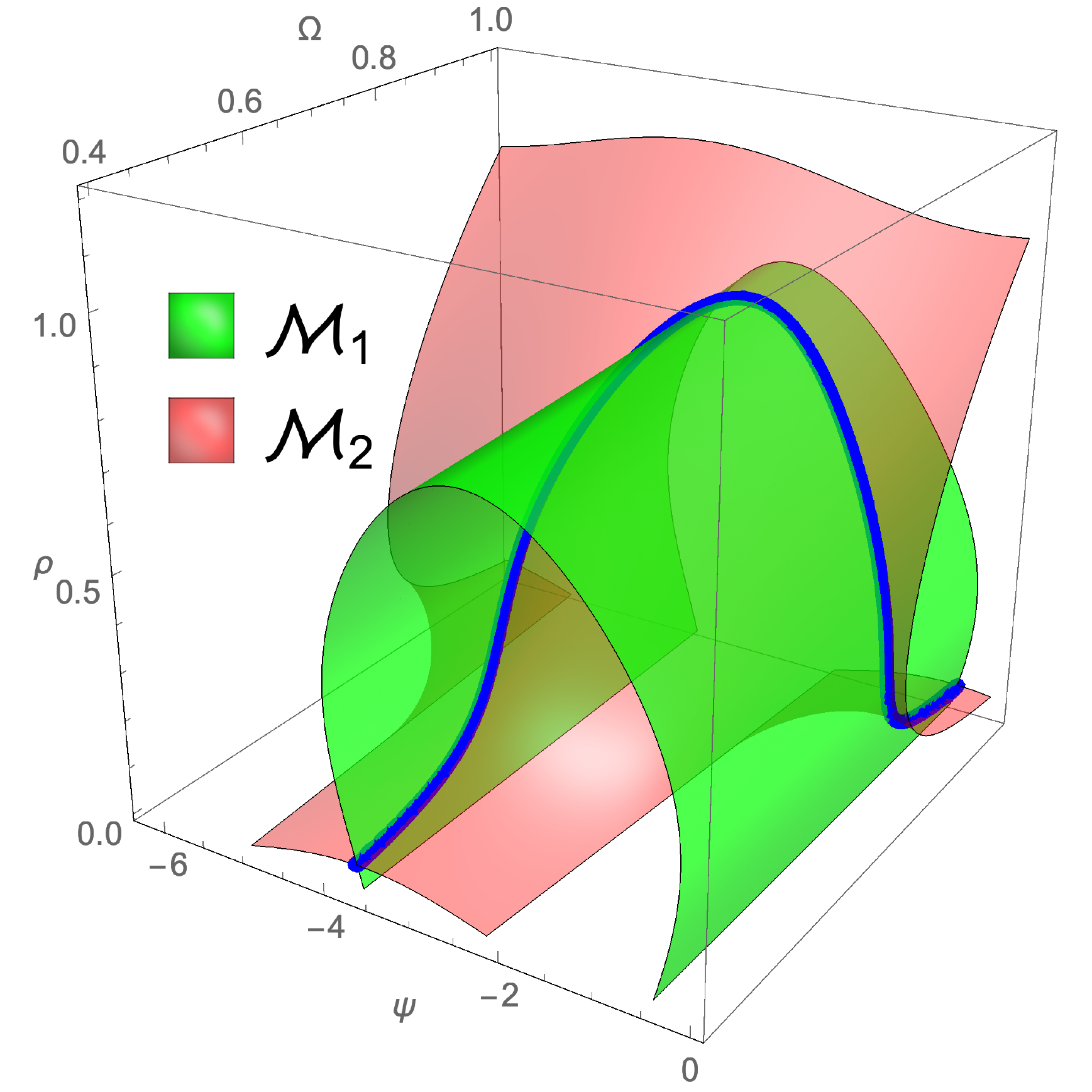}
		 \fi
       \caption{ \label{fig:FRC_intersection_a}}
        \vspace*{2mm}
    \end{subfigure}
    \begin{subfigure}{0.45\textwidth}
        \centering
        \if\mycmd1
        \includegraphics[scale=0.32]{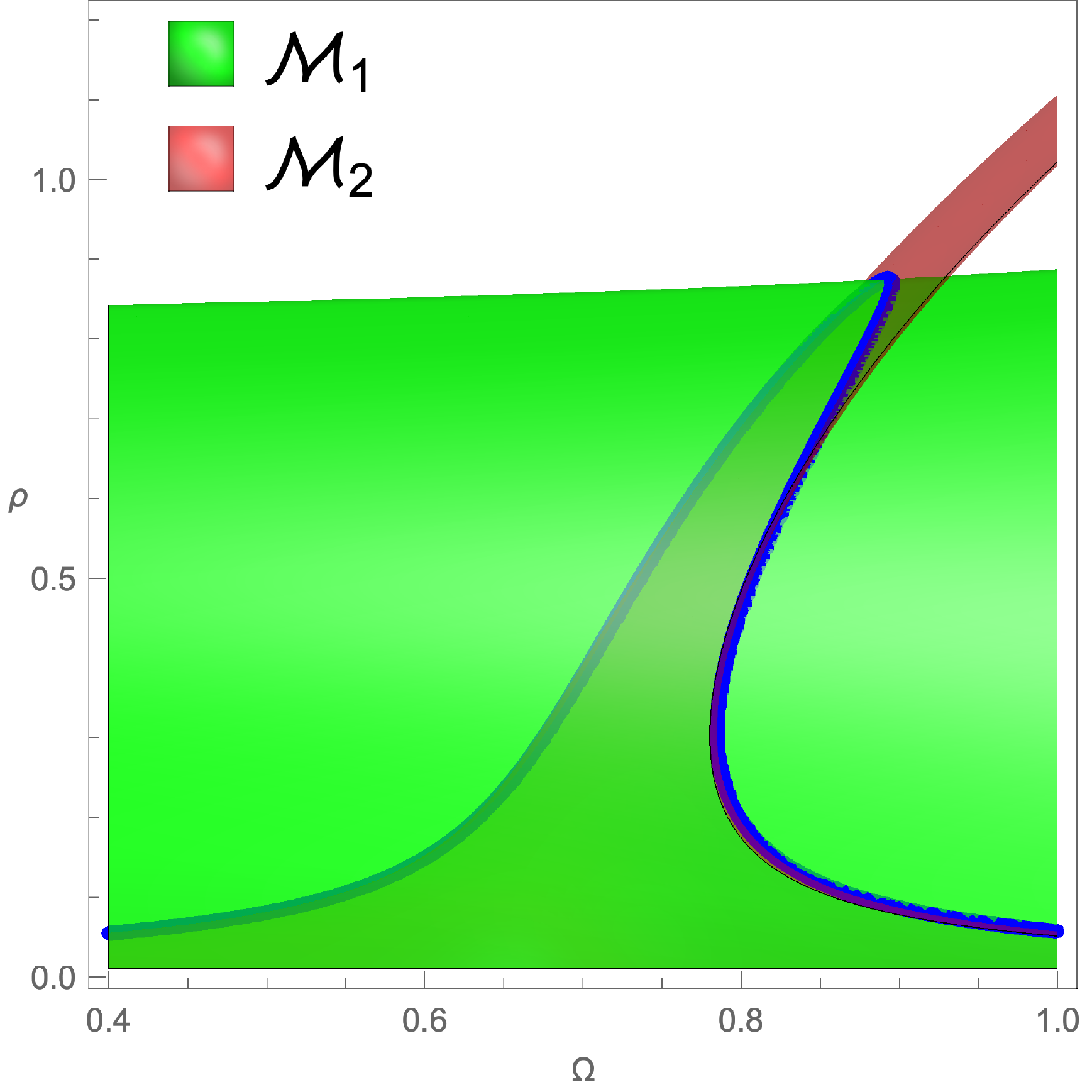}
        \fi
        \caption{ \label{fig:FRC_intersection_b}}
    \end{subfigure}%
    \caption{(a)-(b): Intersection of $\mathcal{M}_1$ (green) and $\mathcal{M}_2$ (red), yielding the frequency-response curve in blue. \label{fig:FRC_intersection}}
\end{figure}

\section{Proof of Theorem \ref{thm:gen_FR} \label{app:impl_eq}}
Let $\mtrx{u}_0$ be a regular point of the map $\mtrx{F}(\mtrx u)$ (\ref{eq:zeroproblem}) such that $\mtrx F(\mtrx u_0)=\mtrx 0$, and $D_\mtrx{u} \mtrx F(\mtrx u_0)$ is surjective. Then, by the implicit function theorem, locally there exists a one-dimensional submanifold of $\mathbb{R}^3$ around $\mtrx u_0$. We express $\psi$ as a function of $\rho$ and $\Omega$, using the tangent half-angle substitution and the trigonometric identities,
\begin{equation}
\frac{\psi}{2} = \tan^{-1}\left(K\right),\quad \cos\left(\psi\right)=\frac{1-K^2}{1+K^2},\quad \sin\left(\psi\right)=\frac{2K}{1+K^2}. \label{eq:trig_id}
\end{equation}
Setting Eq. (\ref{eq:red1_orig}) equal to zero and substituting the identities into Eq. (\ref{eq:red1_orig}), we obtain a quadratic equation in $K$,
\begin{equation}
\left(a(\rho)-\varepsilon f_1(\rho,\Omega)\right)K^2 + 2\varepsilon f_2(\rho,\Omega)K+ \left(a(\rho)+\varepsilon f_1(\rho,\Omega)\right) = 0, \label{eq:quadK}
\end{equation}
which has the solution,
\begin{align}
K(\rho;\Omega)^\pm=\frac{-\varepsilon f_2(\rho,\Omega)\pm \sqrt{\varepsilon^2\left(f_1(\rho,\Omega)^2+f_2(\rho,\Omega)^2\right)-a(\rho)^2}}{a(\rho) - \varepsilon f_1(\rho,\Omega)}. \label{eq:rel_u_v}
\end{align}
Substituting Eq. (\ref{eq:rel_u_v}), together with the trigonometric identities in (\ref{eq:trig_id}), into Eq. (\ref{eq:red2_orig}), we obtain the result stated in Theorem \ref{thm:gen_FR}. \qed

\section{Proof of Theorem \ref{thrm:birth_isola} \label{app:birth_isola}}
We now consider $\varepsilon$ to be a variable in our zero problem (\ref{eq:zeroproblem}), i.e.
\begin{equation}
\mtrx{F}(\mtrx{u},\varepsilon):\mathbb{R}^4 \rightarrow \mathbb{R}^2, \quad
\mtrx{u} = 
\begin{bmatrix}
\rho \\
\Omega\\ 
\psi
\end{bmatrix}. \nonumber
\end{equation}
If there exists a non-spurious non-trivial transverse zero $\rho_0:a(\rho_0)=0$ and $\partial_\rho a(\rho_0)\neq0$, then by restricting ourselves to the autonomous backbone curve (see Ponsioen et al. \cite{ponsioen2018automated}), i.e.
\begin{equation}
\mtrx{u}_0 = [\rho_0,\Omega_0,\psi_0,]^\top,\quad \Omega_0 = b(\rho_0),\quad\psi_0 = \text{const.},\quad \varepsilon = 0, \nonumber
\end{equation}
we have found a solution 
\begin{equation}
\mtrx{F}(\mtrx{u}_0,0) = \mtrx{0}.
\end{equation}
The Jacobian of $\mtrx F$ with respect to $\rho$ and $\Omega$, evaluated at the solution $(\mtrx u_0, 0)$, is given by the square matrix
\begin{align}
D_{(\rho,\Omega)}\mtrx{F}(\mtrx{u}_0,0) = 
\begin{bmatrix}[c]
\partial_\rho a(\rho_0) & 0 \\
\partial_\rho b(\rho_0)\rho_0 & -\rho_0
\end{bmatrix},
\end{align}
which is invertible. Therefore, by the implicit function theorem, we can continue our solution as a two-dimensional submanifold of $\mathbb{R}^4$. Locally, we can express $\rho$ and $\Omega$ as a function of $\psi$ and $\varepsilon$. For $\varepsilon>0$, an isola is born out of the non-trivial transverse zero on the autonomous backbone curve located at $(\Omega,\rho)=(b(\rho_0),\rho_0)$. For a fixed forcing ampltide $\varepsilon$, the isola is parameterized by $\psi$ (as illustrated in Fig. \ref{fig:inter_perb}). \qed

\section{Proof of Theorem \ref{thrm:isola} \label{app:isola}}
In the setting of (\ref{eq:zeroredsys}), our implicit function (\ref{eq:G_thrm}) will reduce to
\begin{equation}
G(\rho,\Omega) = (b(\rho)-\Omega)\rho\pm\sqrt{\varepsilon^2\norm{c_{1,\mtrx{0}}}^2-a(\rho)^2}=0. \label{eq:0_implicit}
\end{equation}
Any zero $\rho_0$ that makes the square root term in Eq. (\ref{eq:0_implicit}) vanish, will also be a zero of (\ref{eq:0_implicit}) itself by setting $\Omega=b(\rho_0)$, and therefore will be on the forced response curve and on the autonomous backbone curve. Additionally, at this point, two segments of the FRC will meet and create a fold over the $\Omega$ direction. We set the argument inside the square root in Eq. (\ref{eq:0_implicit}) equal to zero and rewrite it as
\begin{equation}
\Delta(\rho)=a(\rho)\pm\varepsilon\norm{c_{1,\mtrx{0}}}=0. \label{eq:a_cond}
\end{equation} 
Restricting $\rho\in \mathbb{R}^+_0$, then for $\text{Re}(\gamma_1)>0$, the third-order autonomous function $a(\rho)$ will have a trivial transverse zero and a non-trivial transverse zero located at
\begin{equation}
\rho_0 = 0,\quad \rho_1 = \sqrt{\frac{|\text{Re}(\lambda_1)|}{\text{Re}(\gamma_1)}},
\end{equation}
such that $a(\rho_0)=0$, $a(\rho_1)=0$, $\partial_\rho a(\rho_0)\neq 0$ and $\partial_\rho a(\rho_1)\neq 0$. Under the assumption that the cubic order zero $\rho_1$ is a non-spurious zero for the function $a(\rho)$, then, using the same type of argument as in the proof of Theorem \ref{thrm:birth_isola}, an isola will be born out of this non-trivial transverse zero for system (\ref{eq:zeroredsys}).

We note that between $\rho_0$ and $\rho_1$, $a(\rho)$ will have a local minimum at,
\begin{equation}
\partial_\rho a(\rho) = -|\text{Re}(\lambda_1)|+3\text{Re}(\gamma_1)\rho^2 = 0\quad \Rightarrow \quad \tilde{\rho} = \sqrt{\frac{|\text{Re}(\lambda_1)|}{3\text{Re}(\gamma_1)}}.
\end{equation}
Therefore, for $\varepsilon>0$ small enough, the function $a(\rho)$ will have  three intersections with  the constant curves $\pm \varepsilon \norm{c_{1,\mtrx{0}}}$, meaning that we have found three zeros, $0<\rho_0^\text{a}<\rho_1^\text{a}<\rho_1^\text{b}$ of Eq. (\ref{eq:a_cond}), that correspond to three folding points of the FRC over the $\Omega$ direction. In this setting, $\rho_0^\text{a}$ corresponds to the maximum amplitude of the main FRC branch, $\rho_1^\text{a}$  corresponds to the minimum amplitude of the isola, whereas $\rho_1^\text{b}$ will be the maximum amplitude of the isola.

We can increase $\varepsilon$ such that $\rho_0^\text{a}=\rho_1^\text{a}$, which merges the maximum amplitude of the main FRC branch with the minimum amplitude of the isola, which is exactly at
\begin{equation}
 \varepsilon_\text{m} =  \frac{1}{\norm{c_{1,\mtrx{0}}}}\sqrt{\frac{4|\text{Re}(\lambda_1)|^3}{27\text{Re}(\gamma_1)}}, \label{eq:merge}
\end{equation}
whereas for
\begin{equation}
 0< \varepsilon < \frac{1}{\norm{c_{1,\mtrx{0}}}}\sqrt{\frac{4|\text{Re}(\lambda_1)|^3}{27\text{Re}(\gamma_1)}}. \label{eq:sep}
\end{equation}
the isola will be disconnected from the main FRC, proving Theorem \ref{thrm:isola}. \qed

\bibliographystyle{unsrt} 
\bibliography{ref.bib}

\end{document}